\newcount\ite\ite=1\def\0{\global\ite=1\1}
\def\1{\item{\rm(\romannumeral\the\ite)}\advance\ite1\quad}

\documentclass[12pt]{amsart}
\usepackage{epic, eepic, amsfonts, latexsym, amssymb, graphicx,
multicol, mathrsfs, color, amscd, verbatim, paralist,
xspace, url, euscript, stmaryrd,  amsmath, enumitem,
bbold, multirow, tikz}
\usepackage[all]{xy}

\usepackage[colorlinks, linkcolor=blue, citecolor=magenta, urlcolor=cyan]{hyperref}

\font\teneufm=eufm10 scaled \magstep1
\font\seveneufm=eufm7 scaled \magstep1
\font\fiveeufm=eufm5  scaled \magstep1
\newfam\eufmfam

\textfont\eufmfam=\teneufm
\scriptfont\eufmfam=\seveneufm
\scriptscriptfont\eufmfam=\fiveeufm

\newfam\msbfam
\font\tenmsb=msbm10 scaled \magstep1  \textfont\msbfam=\tenmsb
\font\sevenmsb=msbm7 scaled \magstep1 \scriptfont\msbfam=\sevenmsb
\font\fivemsb=msbm5 scaled \magstep1  \scriptscriptfont\msbfam=\fivemsb

\def\dd#1{\raise1.5pt\hbox{$\,\partial\!$}/\raise-2.5pt\hbox{$\!\partial#1\,$}}

\def\tilde{\widetilde}
\def\hat{\widehat}
\def\5#1{{\mathcal #1}}

\def\RR{{\mathbb R}}
\def\CC{{\mathbb C}}

\def\PP{{\mathbb P}}

\def\ra{\rightarrow}

\def\GL{\mathop{\rm GL}\nolimits}

\def\SO{\mathop{\rm SO}\nolimits}

\def\mod{\mathop{\rm mod}\nolimits}

\def\Im{\mathop{\rm Im}\nolimits}
\def\Re{\mathop{\rm Re}\nolimits}

 \def\HollowBoxx #1#2#3{{\dimen0=#1 \advance\dimen0 by -#2
       \dimen1=#1 \advance\dimen1 by #3
        \vrule height 0pt depth #3 width #2
       \hskip -#3
       \vrule height #1 depth #3 width #3}}
 \def\LeftContraction{\mathord{\kern1.45pt \HollowBoxx{6pt}{3.5pt}{.4pt}}\,}

 \def\HollowBox #1#2#3{{\dimen0=#1 \advance\dimen0 by -#3
       \dimen1=#1 \advance\dimen1 by #3
        \vrule height #1 depth #3 width #3
        \vrule height 0pt depth #3 width #2
        \hskip -#3}}
 \def\RightContraction{\mathord{\, \HollowBox{6pt}{3.1pt}{.4pt}} \kern1.6pt}

\def\qed{{\hfill $\Box$}}
\newtheorem{theorem}{THEOREM}[section]

\theoremstyle{definition}

\newtheorem{lemma}[theorem]{Lemma}

\theoremstyle{remark}
\newtheorem{remark}[theorem]{Remark}

\makeatletter
\def\blfootnote{\xdef\@thefnmark{}\@footnotetext}
\makeatother

\begin{document}

\title[Affine rigidity of tube hypersurfaces]{Affine rigidity of Levi degenerate
\vspace{0.1cm}\\
tube hypersurfaces}
\author[Isaev]{Alexander Isaev}

\address{Mathematical Sciences Institute\\
Australian National University\\
Canberra, ACT 0200, Australia}
\email{alexander.isaev@anu.edu.au}

\maketitle

\thispagestyle{empty}

\pagestyle{myheadings}

\begin{abstract}
Let ${\mathfrak C}_{2,1}$ be the class of connected 5-dimensional CR-hypersurfaces that are 2-nondegenerate and uniformly Levi degenerate of rank 1. In the recent article \cite{IZ}, we proved that the CR-structures in ${\mathfrak C}_{2,1}$ are reducible to ${\mathfrak{so}}(3,2)$-valued absolute parallelisms. In the present paper, we apply this result to study tube hypersurfaces in $\CC^3$ that belong to ${\mathfrak C}_{2,1}$ and whose CR-curvature identically vanishes. Every such hypersurface is shown to be affinely equivalent to an open subset of the tube over the future light cone $\left\{(x_1,x_2,x_3)\in\RR^3\mid x_1^2+x_2^2-x_3^2=0,\,\ x_3>0\right\}$.
\end{abstract}

\setcounter{section}{0}

\section{Introduction}\label{intro}
\setcounter{equation}{0}

We consider connected smooth real hypersurfaces in the complex vector space $\CC^n$ with $n\ge 2$. Specifically, we are interested in {\it tube hypersurfaces}, i.e.~locally closed real submanifolds of the form
$$
{\mathcal S}+iV,
$$
where ${\mathcal S}$ is a hypersurface in a totally real $n$-dimensional linear subspace $V\subset\CC^n$. One can choose coordinates $z_1,\dots,z_n$ in $\CC^n$ such that $V=\{\Im z_j=0,\,\,j=1,\dots,n\}$, and everywhere below $V$ is identified with $\RR^n$ by means of the coordinates $x_j:=\Re z_j$, $j=1,\dots,n$. Tube hypersurfaces arise, for instance, as the boundaries of tube domains, that is, domains of the form
$$
{\mathcal D}+i\RR^n,
$$
where ${\mathcal D}$ is a domain in $\RR^n$. We refer to the hypersurface ${\mathcal S}$ and domain ${\mathcal D}$ as the {\it bases}\, of the above tubes. 

The study of tube domains is a classical subject in several complex variables and complex geometry, which goes back to the beginning of the 20th century. Indeed, already Siegel found it convenient to realize certain symmetric domains as tubes. For example, the familiar unit ball in $\CC^n$ is biholomorphically equivalent to the tube domain with the base given by
$$
x_n>\sum_{\alpha=1}^{n-1}x_{\alpha}^2.
$$

The property that makes tube domains interesting from the complex-geometric point of view is that they possess an $n$-dimensional commutative group of holomorphic symmetries, namely the group of translations $\{z\mapsto z+ib\mid b\in\RR^n\}$, where $z:=(z_1,\dots,z_n)\in\CC^n$. Furthermore, any affine automorphism of the base of a tube can be extended to a holomorphic affine automorphism of the tube. In the same way, any affine transformation between the bases of two tubes can be lifted to a holomorphic affine transformation between the tubes. This last observation, however simple, indicates an important link between complex and affine geometries by way of tube domains. Indeed, they can be viewed either as objects of affine geometry and considered up to affine transformations of the bases or as objects of complex geometry and considered up to biholomorphisms, and the first collection of maps is included in the second one. In the present paper we take a similar approach to tube hypersurfaces and look at them from both the affine-geometric and CR-geometric points of view. We will now proceed with describing the background and content of the article and refer the reader to Section \ref{model} for all necessary definitions and facts from CR-geometry.

There has been a substantial effort to relate the two aspects of the study of tubes (see, e.g., \cite{Ma}, \cite{Sh1}, \cite{Sh2}, \cite{Lo}, \cite{EEI},  \cite{KS1}, \cite{KS2} \cite{FK3}, \cite{FK4}, \cite{I1}). In particular, one would like to understand the interplay between affine equivalence and CR-equivalence for tube hypersurfaces, where $M_1={\mathcal S}_1+i\RR^n$ and $M_2={\mathcal S}_2+i\RR^n$ are called {\it affinely equivalent}\, if there exists an affine transformation of $\CC^n$ of the form
\begin{equation}
z\mapsto Az+b,\quad A\in\GL_n(\RR),\quad b\in\CC^n\label{affequiv}
\end{equation}
that maps $M_1$ onto $M_2$ (this occurs if and only if the bases ${\mathcal S}_1$ and ${\mathcal S}_2$ are affinely equivalent as submanifolds of $\RR^n$). Specifically, the following two questions have attracted much attention:
\vspace{-0.8cm}\\

$$
\begin{array}{l}
\hspace{0.1cm}\hbox{$(*)$ When does local or global CR-equivalence of $M_1$, $M_2$ imply}\\
\vspace{-0.3cm}\hspace{0.8cm}\hbox{affine equivalence?}\\
\vspace{-0.1cm}\\
\hbox{$(**)$ For what kinds of tube hypersurfaces can one refine known}\\
\vspace{-0.3cm}\hspace{0.9cm}\hbox{CR-classification results to deduce affine classifications?} 
\end{array}
$$
\vspace{-0.4cm}\\

So far, the most comprehensive answers to the above questions have been given for the class of Levi nondegenerate tube hypersurfaces that are {\it CR-flat}, i.e.~have identically vanishing CR-curvature, as defined below. For a CR-hypersurface $M$ with Levi form of signature $(p,q)$, where $p+q=n-1$, $q\le p$, the condition of CR-flatness means that near every point $M$ is CR-equivalent to an open subset of the real affine quadric
$$
\Re z_n=\sum_{\alpha=1}^p|z_{\alpha}|^2-\sum_{\alpha=p+1}^{n-1}|z_{\alpha}|^2.
$$
Thus, for a fixed signature $(p,q)$ of the Levi form, all CR-flat tube hypersurfaces are pairwise locally CR-equiva\-lent to each other. Nevertheless, the question of classifying them up to affine equivalence is highly nontrivial. We refer the reader to monograph \cite{I1} for an up-to-date exposition of the existing theory. In particular, there are explicit affine classifications for $q=0,1,2$ and an overall understanding of how results of this kind can be obtained in general. Moreover, CR-flat tube hypersurfaces possess remarkable properties. In particular, G. Fels and W. Kaup have discovered a deep connection between such hypersurfaces and commutative algebra by showing that all CR-flat tube hypersurfaces arise, in a canonical way, from real and complex Artinian Gorenstein algebras (see \cite{FK3} and Chapter 9 of \cite{I1} for details). This last fact has led to an independent line of research concerning Artinian Gorenstein algebras over arbitrary fields (see \cite{FK5}, \cite{I2}) and even inspired a new approach to constructing invariants of isolated hypersurface singularities (see \cite{FIKK}, \cite{EI}, \cite{AI}).

Our aim is to extend these results to other classes of CR-flat tube hypersurfaces. In general, CR-curvature is defined in situations when the CR-structures in question are reducible to absolute parallelisms with values in some Lie algebra ${\mathfrak g}$. Let ${\mathfrak C}$ be a class of CR-manifolds. Then the CR-structures in ${\mathfrak C}$ are said to reduce to ${\mathfrak g}$-valued absolute parallelisms if to every $M\in{\mathfrak C}$ one can assign a fiber bundle\linebreak ${\mathcal P}_M\ra M$ and an absolute parallelism $\omega_M$ on ${\mathcal P}_M$ such that for every $p\in M$ the parallelism establishes an isomorphism between $T_p(M)$ and ${\mathfrak g}$ and for any $M_1,M_2\in{\mathfrak C}$ the following holds: 

\noindent (i) every CR-isomorphism $f:M_1\ra M_2$ can be lifted to a diffeomorphism $F: {\mathcal P}_{M_{{}_1}}\ra{\mathcal P}_{M_{{}_2}}$ satisfying
\begin{equation}
F^{*}\omega_{M_{{}_2}}=\omega_{M_{{}_1}},\label{eq8}
\end{equation}
and 

\noindent (ii) any diffeomorphism $F: {\mathcal P}_{M_{{}_1}}\ra{\mathcal P}_{M_{{}_2}}$ satisfying (\ref{eq8}) 
is a bundle isomorphism that is a lift of a CR-isomorphism $f:M_1\ra M_2$. In this situation one introduces the ${\mathfrak g}$-valued {\it CR-curvature form}\,
\begin{equation}
\Omega_M:=d\omega_M-\frac{1}{2}\left[\omega_M,\omega_M\right],\label{genformulacurvature}
\end{equation}
and CR-flatness means that $\Omega_M$ identically vanishes on ${\mathcal P}_M$.
 
Reducing CR-structures (as well as other geometric structures) to absolute parallelisms goes back to \'E. Cartan  who showed that reduction takes place for all 3-dimensional Levi nondegenerate CR-hyper\-surfaces (see \cite{Ca}). Since then there have been many developments (see, e.g., \cite{T1}, \cite{T2}, \cite{T3}, \cite{CM}, \cite{Ch}, \cite{BS1}, \cite{BS2}, \cite{J}, \cite{La}, \cite{Mi}, \cite{GM}, \cite{EIS}, \cite{CSc}, \cite{CSl}, \cite{ScSl}, \cite{ScSp}), all of which require Levi nondegeneracy. In particular, the famous work of Tanaka and, independently, that of Chern and Moser established reduction to absolute parallelisms for all Levi nondegenerate CR-hypersurfaces. On the other hand, reducing the CR-structures of Levi degenerate CR-mani\-folds has proved to be rather difficult. Despite \'E. Cartan's approach having been known for over a century and Tanaka's work published almost half a century ago, the first result on reduction to absolute parallelisms for a large class of Levi degenerate manifolds only appeared in 2013 in our paper \cite{IZ}. Specifically, we considered the class ${\mathfrak C}_{2,1}$ of connected 5-dimensional CR-hypersurfaces that are 2-nondegenerate and uniformly Levi degenerate of rank 1 (see Section \ref{model} for definitions) and showed that the CR-structures in this class reduce to ${\mathfrak{so}}(3,2)$-valued parallelisms. Reduction to absolute parallelisms for this class was attempted earlier in article \cite{E}, but the proof contained a serious flaw (see the correction to \cite{E}). Yet another construction for ${\mathfrak C}_{2,1}$ was presented in the recent preprint \cite{MS}. However, the proofs given in \cite{MS} are mere sketches and we have been unable to verify all the details. A related reduction result was also obtained in preprint \cite{P} for hypersurfaces in $\CC^3$ with nowhere vanishing CR-curvature (see also Remark \ref{twofunctions}).

As explained in \cite{IZ}, a manifold $M\in{\mathfrak C}_{2,1}$ is CR-flat if and only if near every point $M$ is CR-equivalent to an open subset of the tube hypersurface over the future light cone in $\RR^3$:
\begin{equation}
M_0:=\left\{(z_1,z_2,z_3)\in\CC^3\mid x_1^2+x_2^2-x_3^2=0,\,\ x_3>0\right\}.\label{light}
\end{equation}
Note that this hypersurface had been extensively studied prior to our work (see, e.g., \cite{FK1}, \cite{FK2}, \cite{KZ}, \cite{Me}). In particular, it can be realized as part of the boundary of the classical symmetric domain of type $({\rm IV}_3)$, or, equivalently, of type $({\rm III}_2)$ (see Section \ref{model} for details). 

We are now ready to state the main theorem of this paper.

\begin{theorem}\label{main2}
Let $M$ be a tube hypersurface in $\CC^3$ and assume that $M\in{\mathfrak C}_{2,1}$. Suppose further that $M$ is CR-flat. Then $M$ is affinely equivalent to an open subset of $M_0$. 
\end{theorem}

\noindent This theorem can be viewed as an affine rigidity result since it asserts that if a tube hypersurfaces $M$ in the class ${\mathfrak C}_{2,1}$ locally looks like a piece of $M_0$ from the point of view of CR-geometry, then from the point of view of affine geometry it (globally) looks like a piece of $M_0$ as well. In fact, $M$ extends to a tube hypersurface in $\CC^3$ affinely equivalent to $M_0$,  which is a complete answer to question $(*)$ in this situation.

The paper is organized as follows. In Section \ref{model} we give necessary CR-geometric preliminaries. Section \ref{construction} contains an outline of our construction from \cite{IZ}. In Section \ref{calculations} we perform a detailed analysis of the invariants described in Section \ref{construction} in terms of a local defining function of a tube hypersurface. In particular, it turns out that in order to obtain the conclusion of Theorem \ref{main2}, one does not need to assume that the full curvature form is identically zero; it suffices to require that only two coefficients in the expansion of just one of its components vanish. This stronger result is stated as Theorem \ref{main1}.

One observation critical for the proof of Theorem \ref{main1} is the fact that the base of any tube hypersurface $M$ in $\CC^3$ that is uniformly Levi degenerate of rank 1 can be locally written as the graph of a function of two variables satisfying the homogeneous Monge-Amp\`ere equation. It is well-known that the general solution of this equation is given explicitly in parametric form in terms of two arbitrary functions of one variable. For $M\in{\mathfrak C}_{2,1}$, the vanishing of the two curvature coefficients assumed in Theorem \ref{main1} imposes conditions on these two functions that translate into local (hence global) affine equivalence of $M$ to an open subset of $M_0$. Overall, the proof of Theorem \ref{main1} in Section \ref{calculations} shows that the reduction to absolute parallelisms achieved in \cite{IZ} can be successfully applied to calculations in terms of defining functions. This outcome is rather encouraging and calls for further applications.

{\bf Acknowledgements.} The research is supported by the Australian Research Council. This work was finalized during the author's visit to the Max-Planck Institute for Mathematics in Bonn in 2014.

\section{Preliminaries}\label{model}
\setcounter{equation}{0}

Recall that an almost CR-structure on a smooth manifold $M$ is a subbundle $H(M)\subset T(M)$ of the tangent bundle of even rank endowed with operators of complex structure $J_p:H_p(M)\ra H_p(M)$, $J_p^2= -\hbox{id}$, that smoothly depend on $p\in M$. A manifold equipped with an almost CR-structure is called an almost CR-manifold. The subspaces $H_p(M)$ are called the complex tangent spaces to $M$, and their complex dimension, denoted by $\hbox{CRdim}\, M$, is the CR-dimension of  $M$. The complementary dimension  $\dim M-2\hbox{CRdim}\, M$ is called the CR-codimension of $M$. 

Further, a smooth map $f:M_1\ra M_2$ between two almost CR-manifolds is a CR-map if for every $p\in M_1$ the differential $df(p)$ of $f$ at $p$ maps $H_p(M_1)$ into $H_{f(p)}(M_2)$ and is complex-linear on $H_p(M_1)$. If for two almost CR-manifolds $M_1$, $M_2$ of equal CR-dimensions there exists a diffeomorphism $f$ from $M_1$ onto $M_2$ that is also a CR map, then the manifolds are said to be CR-equivalent and $f$ is called a CR-isomorphism.

Next, for every $p\in M$ consider the complexification
$H_p(M)\otimes_{\RR}\CC$ of the complex tangent space at $p$. It can be
represented as the direct sum
$$
H_p(M)\otimes_{\RR}\CC=H_p^{(1,0)}(M)\oplus H_p^{(0,1)}(M),
$$
where
$$
\begin{array}{l}
H_p^{(1,0)}(M):=\{X-iJ_pX\mid X\in H_p(M)\},\\
\vspace{-0.3cm}\\
H_p^{(0,1)}(M):=\{X+iJ_pX\mid X\in H_p(M)\}.
\end{array}
$$
Then the almost CR-structure on $M$ is said to be integrable if the bundle $H^{(1,0)}$ is involutive, i.e.~for any pair of local
sections ${\mathfrak z},{\mathfrak z}'$ of $H^{(1,0)}(M)$ the commutator $[{\mathfrak z},{\mathfrak z}']$
is also a local section of $H^{(1,0)}(M)$. An integrable almost CR-structure is called a CR-structure and a manifold equipped with a CR-structure a CR-manifold. In this paper we consider only CR-hypersurfaces, i.e.~CR-manifolds of CR-codimension 1.

If $M$ is a real hypersurface in a complex manifold $X$ with operators of almost complex structure ${\mathcal J}_q$, $q\in X$,  it is naturally an almost CR-manifold with $H_p(M):=T_p(M)\cap {\mathcal J}(T_p(M))$ and $J_p$ being the restriction of ${\mathcal J}_p$ to $H_p(M)$ for every $p\in M$. Furthermore, the almost complex structure so defined is integrable, so $M$ is in fact a CR-hypersurface of CR-dimension $\dim X-1$. In particular, any tube hypersurface in $\CC^n$ is a CR-hypersurface of CR-dimension $n-1$.

Further, the Levi form of a CR-hypersurface $M$ comes from taking commutators of local sections of
$H^{(1,0)}(M)$ and $H^{(0,1)}(M)$. Let $p\in M$, $\zeta,\zeta'\in
H_p^{(1,0)}(M)$. Choose  local sections ${\mathfrak z}$, ${\mathfrak z}'$ of $H^{(1,0)}(M)$ near
$p$ such that ${\mathfrak z}(p)=\zeta$, ${\mathfrak z}'(p)=\zeta'$. The Levi form of $M$ at
$p$ is then the Hermitian form on $H_p^{(1,0)}(M)$ with values in $(T_p(M)/H_p(M))\otimes_{\RR}\CC$ given by
$$
{\mathcal L}_M(p)(\zeta,\zeta'):=i[{\mathfrak z},\overline{{\mathfrak z}'}](p)(\mod H_p(M)\otimes_{\RR}\CC).
$$
For fixed $\zeta$ and $\zeta'$ the right-hand side of the above formula is independent of the choice of ${\mathfrak z}$ and ${\mathfrak z}'$. Identifying $T_p(M)/H_p(M)$ with $\RR$, one obtains a $\CC$-valued Hermitian form defined up to a real scalar multiple.

In this paper, we consider 5-dimen\-sional CR-hypersurfaces that are uniformly Levi degenerate of rank 1. This means that the kernel $\ker{\mathcal L}_M(p)$ of the Levi form has dimension 1 at every $p\in M$, where 
$$
\ker{\mathcal L}_M(p):=\left\{\zeta\in H_p^{(1,0)}(M)\mid {\mathcal L}_M(p)(\zeta,\zeta')=0\, \forall\, \zeta'\in H_p^{(1,0)}(M)\right\}.
$$

We will now discuss the condition of 2-nondegeneracy. For the general notion of $k$-nondegeneracy (as well as other nondegeneracy conditions) we refer the reader to Chapter XI in \cite{BER} and note that for CR-hypersurfaces 1-nondegeneracy is equivalent to Levi nondegeneracy. Rather than giving the general definition of 2-nondegeneracy, we explain what this condition means in the case at hand.

Let $M$ be a 5-dimensional CR-hypersurface uniformly Levi degenerate of rank 1. Fix $p_0\in M$. Locally near $p_0$ the CR-structure is given by 1-forms $\mu$, $\eta^{\alpha}$, $\alpha=1,2$, where $\mu$ is $i\RR$-valued and vanishes precisely on the complex tangent spaces $H_p(M)$, and $\eta^{\alpha}$ are $\CC$-valued and their restrictions to $H_p(M)$ at every point $p$ are $\CC$-linear and constitute a basis of $H_p^*(M)$. The integrability condition for the CR-structure is then equivalent to the Frobenius condition, which states that $d\mu$, $d\eta^{\alpha}$ belong to the differential ideal generated by $\mu$, $\eta^{\beta}$ (see, e.g., pp.~174--175 in \cite{Y}). Since $\mu$ is $i\RR$-valued, this implies
\begin{equation}
d\mu\equiv h_{\alpha\overline{\beta}}\eta^{\alpha}\wedge\eta^{\overline{\beta}}\quad(\mod \mu)\label{integr0}
\end{equation}
for some functions $h_{\alpha\overline{\beta}}$ satisfying $h_{\alpha\overline{\beta}}=h_{\overline{\beta}\alpha}$, where we use the summation convention for subscripts and superscripts (here and everywhere below the conjugation of indices denotes the conjugation of the corresponding forms, e.g., $\eta^{\bar\beta}:=\overline{\eta^{\beta}}$). Since $M$ is uniformly Levi degenerate of rank 1, one can choose $\eta^{\alpha}$ near $p_0$ so that 
\begin{equation}
(h_{\alpha\overline{\beta}})\equiv\left(
\begin{array}{cc}
\pm 1 & 0\\
0 & 0
\end{array}
\right).\label{hform}
\end{equation}

Next, the integrability condition yields
$$
d\eta^1\equiv \eta^2\wedge\sigma\quad(\mod \mu,\eta^1)
$$
for some complex-valued 1-form $\sigma$. Now, assuming that (\ref{hform}) holds, we say that $M$ is 2-nondegenerate at $p_0$ if the coefficient at $\eta^{\bar 1}$ in the expansion of $\sigma$ with respect to $\mu$, $\eta^{\alpha}$, $\eta^{\bar\alpha}$ does not vanish at $p_0$. Clearly, with (\ref{hform}) satisfied, this condition is independent of the choice of $\mu$, $\eta^{\alpha}$. Also, we say that $M$ is 2-nondegenerate if $M$ is 2-nondegenerate at every point. As shown in \cite{E} (see Proposition 1.16 and p. 51 therein), this definition of 2-nondegeneracy is equivalent to the standard one. In our arguments in the forthcoming sections we utilize the definition given above.

Recall that ${\mathfrak C}_{2,1}$ was defined in the introduction as the class of connected 5-dimensional CR-hyper\-surfaces that are 2-non\-dege\-nerate and uniformly Levi degenerate of rank 1. It is not hard to see that the tube hypersurface $M_0$ introduced in (\ref{light}) belongs to this class. In the next section we will outline the procedure from \cite{IZ} for reducing the CR-structures in ${\mathfrak C}_{2,1}$ to absolute parallelisms. The hypersurface $M_0$ (or, rather, a certain extension of $M_0$) serves as a model for the reduction, and we will now briefly discuss it.

For $x=(x_1,\dots,x_5)\in\RR^5$ set
$$
(x,x):=x_1^2+x_2^2+x_3^2-x_4^2-x_5^2, 
$$
and realize $\SO(3,2)$ as the group of real $5\times 5$-matrices $C$ with $\det C=1$ satisfying $(Cx,Cx)\equiv (x,x)$. Consider the symmetric and Hermitian extensions of the form $(\,\,,\,)$ to $\CC^5$. Denote the symmetric extension by the same symbol $(\,\,,\,)$ and the Hermitian extension by $\langle\,\,,\,\rangle$. For $Z=(z_1:\dots:z_5)\in\CC\PP^4$ we now consider the projective quadric
$$
{\mathcal Q}:=\{Z\in\CC\PP^4\mid (Z,Z)=0\}
$$
and the open subset $D\subset{\mathcal Q}$ defined as
$$
D:=\{Z\in{\mathcal Q}\mid \langle Z,Z\rangle<0\}.
$$
Observe that the group $\SO(3,2)$ acts on $D$, and the action is effective.

The set $D$ has two connected components as follows:
$$
\begin{array}{l}
D_{+}:=\left\{Z\in D\mid \Re z_4\Im z_5-\Re z_5\Im z_4>0\right\},\\
\vspace{-0.3cm}\\
D_{-}:=\left\{Z\in D\mid\Re z_4\Im z_5-\Re z_5\Im z_4<0\right\}.
\end{array}
$$
Each of these components is a realization of the classical symmetric domain of type $({\rm IV}_3)$, or, equivalently, of type $({\rm III}_2)$, and all holomorphic automorphisms of each of $D_{\pm}$ arise from the action of the group ${\mathcal G}:=\SO(3,2)^{\circ}$, the connected identity component of $\SO(3,2)$  (see pp. 285--289 in \cite{Sa}).

Further, it is not hard to see that the action of ${\mathcal G}$ on $\partial D_{+}\cup\partial D_{-}\subset{\mathcal Q}$ has two real hypersurface orbits:
$$
\begin{array}{l}
\Gamma_{+}:=\left\{Z\in{\mathcal Q}\mid (\Re z,\Re z)=(\Im z,\Im z)=(\Re z,\Im z)=0,\right.\\
\vspace{-0.3cm}\\
\hspace{1.05cm}\left.\Re z_4\Im z_5-\Re z_5\Im z_4>0\right\}\subset\partial D_{+},\\
\vspace{-0.1cm}\\
\Gamma_{-}:=\left\{Z\in{\mathcal Q}\mid (\Re z,\Re z)=(\Im z,\Im z)=(\Re z,\Im z)=0,\right.\\
\vspace{-0.3cm}\\
\hspace{1.05cm}\left.\Re z_4\Im z_5-\Re z_5\Im z_4<0\right\}\subset\partial D_{-}.
\end{array}
$$
These hypersurfaces are CR-equivalent, and we will only consider $\Gamma_{+}$, which is the orbit of the point $q_{+}:=(i:1: 0:1: i)$. Writing $D_{+}$ in tube form (see p. 289 in \cite{Sa} and p.~64 in \cite{FK1}), one observes that $M_0$ is CR-equivalent to an open dense subset of $\Gamma_{+}$. Since $\Gamma_{+}$ is homogeneous, we then see that it belongs to the class ${\mathfrak C}_{2,1}$. By Theorems 4.5, 4.7 of \cite{KZ}, in this realization every local CR-automorphism of $M_0$ extends to a holomorphic automorphism of ${\mathcal Q}$ induced by an element of ${\mathcal G}$.

We will now bring the hypersurface $\Gamma_{+}$ to a form better adapted for our future calculations. Let $\Phi:(z_1:\dots:z_5)\mapsto(z_1^*:\dots:z_5^*)$ be the automorphism of $\CC\PP^4$ given by
$$
\begin{array}{lll}
\displaystyle z_1^*=\frac{1}{2}(z_1+iz_2-iz_4-z_5), & \displaystyle z_2^*=\frac{1}{2}(z_1-iz_2+iz_4-z_5), & z_3^*=z_3,\\
\vspace{-0.1cm}\\
\displaystyle z_4^*=\frac{1}{2}(z_1+iz_2+iz_4+z_5), & \displaystyle z_5^*=\frac{1}{2}(z_1-iz_2-iz_4+z_5).
\end{array}
$$
When viewed as a transformation of $\CC^5$, it takes $(\,\,,\,)$ and $\langle\,\,,\,\rangle$ into the bilinear and Hermitian forms defined, respectively, by the following matrices:
$$
{\mathbf S}:=\left(
\begin{array}{lllll}
0 & 0 & 0&0& 1\\
0 & 0 & 0&1& 0\\
0 & 0 & 1&0& 0\\
0 & 1 & 0&0& 0\\
1 & 0 & 0&0& 0\\
\end{array}
\right),\quad
{\mathbf T}:=\left(
\begin{array}{lllll}
0 & 0 & 0&1& 0\\
0 & 0 & 0&0& 1\\
0 & 0 & 1&0& 0\\
1 & 0 & 0&0& 0\\
0 & 1 & 0&0& 0\\
\end{array}
\right).
$$
Let $G$ be the connected identity component of the group of complex $5\times 5$-matrices $C$ with $\det C=1$ satisfying
$$
C^t{\mathbf S}C = {\mathbf S}, \quad C^t{\mathbf T}\bar C ={\mathbf T}
$$
(clearly, $G$ is isomorphic to ${\mathcal G}$). The Lie algebra ${\mathfrak g}$ of $G$ (which is isomorphic to ${\mathfrak{so}}(3,2)$) consists of all matrices of the form
\begin{equation}
\left(\begin{array}{rrrrr}
\alpha & \beta & \gamma & \delta & 0\\
\bar \beta & \bar \alpha & \bar \gamma & 0 & -\delta\\
\sigma & \bar \sigma & 0 & -\bar\gamma & -\gamma\\
\rho & 0 & -\bar\sigma & -\bar\alpha & -\beta\\
0 & -\rho & -\sigma & -\bar\beta & -\alpha
\end{array}
\right),\label{liealgebra}
\end{equation}
with $\alpha,\beta,\gamma,\sigma\in\CC$, $\delta,\rho\in i\RR$. 

Set $\Gamma:=\Phi(\Gamma_{+})$, $q:=\Phi(q_{+})=(0:0:0:1:0)$. Then $\Gamma=G\cdot q$, and we denote by $H\subset G$ the isotropy subgroup of $q$. One has $H=H^1\ltimes H^2$, where $H^1$, $H^2$ are the following subgroups of $G$:
\begin{equation}
H^1:=\left\{\left(
\begin{array}{ccccc}
A & 0 & 0&0& 0\\
0 & \bar A & 0&0& 0\\
0 & 0 & 1&0& 0\\
0 & 0 & 0&\bar A^{-1}& 0\\
0 & 0 & 0&0&A^{-1}\\
\end{array}
\right),\,A\in\CC^*\right\},\label{subgroups1}
\end{equation}
\begin{equation}
H^2:=\left\{\left(
\begin{array}{ccccc}
1 & 0 & 0&0& 0\\
0 & 1 & 0&0& 0\\
B & \bar B & 1&0& 0\\
\Lambda-|B|^2/2 &\hspace{-0.5cm} -\bar B^2/2 & -\bar B&1& 0\\
\hspace{-0.5cm}-B^2/2 & \hspace{-0.5cm}-\Lambda-|B|^2/2 & - B&0&1\\
\end{array}
\right),\begin{array}{l}B\in\CC,\\ \Lambda\in i\RR\end{array}\right\}.\label{subgroups2}
\end{equation}
The hypersurface $\Gamma$ will turn out to be a CR-flat model for our reduction to parallelisms in the next section, and the subgroups $H^1$, $H^2$ will be instrumental in the reduction process.

\section{Reduction to parallelisms}\label{construction}
\setcounter{equation}{0}

In this section we outline the procedure from \cite{IZ} for reducing the CR-structures in ${\mathfrak C}_{2,1}$ to ${\mathfrak g}$-valued absolute parallelisms. Our proof of Theorem \ref{main2} in the next section is based on this procedure.

Fix $M\in{\mathfrak C}_{2,1}$. As mentioned in Section \ref{model}, locally on $M$ the CR-structure is given by 1-forms $\mu$, $\eta^{\alpha}$, $\alpha=1,2$, where $\mu$ is $i\RR$-valued and vanishes precisely on the complex tangent spaces $H_p(M)$, and $\eta^{\alpha}$ are $\CC$-valued and their restrictions to $H_p(M)$ at every point $p$ are $\CC$-linear and form a basis of $H_p^*(M)$. The integrability condition for the CR-structure then implies that identity (\ref{integr0}) holds for some functions $h_{\alpha\overline{\beta}}$ satisfying $h_{\alpha\overline{\beta}}=h_{\overline{\beta}\alpha}$.

For $p\in M$ define $E_p$ as the collection of all $i\RR$-valued covectors $\theta$ on $T_p(M)$ such that $H_p(M)=\{X\in T_p(M)\mid \theta(X)=0\}$. Clearly, all elements in $E_p$ are real nonzero multiples of each other. Let $E$ be the bundle over $M$ with fibers $E_p$. Define $\omega$ to be the tautological 1-form on $E$, that is, for $\theta\in E$ and $X\in T_{\theta}\left(E\right)$ set
$$
\omega(\theta)(X):=\theta(d\pi_E(\theta)(X)),
$$
where $\pi_E: E\ra M$ is the projection. Since the Levi form of $M$ has rank 1 everywhere, identity (\ref{integr0}) implies that for every $\theta\in E$ there exist a real-valued covector $\phi$ and a complex-valued covector $\theta^1$ on $T_{\theta}(E)$ such that: (i) $\theta^1$ is the pull-back of a complex-valued covector on $T_{\pi_{{}_E}(\theta)}(M)$ complex-linear on $H_{\pi_{{}_E}(\theta)}(M)$, and (ii) the following identity holds:
\begin{equation}
d\omega(\theta)=\pm\theta^1\wedge\theta^{\overline{1}}-\omega(\theta)\wedge\phi.\label{integr}
\end{equation}

For every $p\in M$ the fiber $E_p$ has exactly two connected components, and the signs in the right-hand side of (\ref{integr}) coincide for all covectors $\theta$ lying in the same connected component of $E_p$ and are opposite for any two covectors lying in different connected components irrespectively of the choice of $\theta^1$, $\phi$. We then define a bundle ${\mathcal P}^1$ over $M$ as follows: for every $p\in M$ the fiber ${\mathcal P}^1_p$ over $p$ is connected and consists of all elements $\theta\in E_p$ for which the minus sign occurs in the right-hand side of (\ref{integr}); we also set $\pi^1:=\pi_E\bigr|_{{\mathcal P}^1}$. 

Next, the most general transformation of $(\omega(\theta),\theta^1,\theta^{\overline{1}},\phi)$ preserving the equation 
\begin{equation}
d\omega(\theta)=-\theta^1\wedge\theta^{\overline{1}}-\omega(\theta)\wedge\phi\label{integrminus}
\end{equation}
and the covector $\omega(\theta)$ is  given by the matrix (acting on the left)
\begin{equation}
\left(
\begin{array}{rccc}
1 & 0 & 0 & 0\\
\vspace{-0.3cm}\\
\overline{b} & a & 0 & 0\\
\vspace{-0.3cm}\\
-b & 0 & \overline{a} & 0\\
\vspace{-0.3cm}\\
\lambda & -ab & -\overline{a}\overline{b} & 1
\end{array}
\right),\label{g1structure}
\end{equation}
where $a,b\in\CC$, $|a|=1$, $\lambda\in i\RR$. Let $H_1$ be the group of matrices of the form (\ref{g1structure}). Observe that $H_1$ is isomorphic to the subgroup $H_1^1\ltimes H^2$ of $H$, where $H_1^1$ is the subgroup of $H^1$ given by the condition $|A|=1$ (see (\ref{subgroups1}), (\ref{subgroups2})). Our goal is to reduce the $H_1$-structure on ${\mathcal P}^1$ to an absolute parallelism.

We now introduce a principal $H_1$-bundle ${\mathcal P}^2$ over ${\mathcal P}^1$ as follows: for $\theta\in {\mathcal P}^1$ let the fiber ${\mathcal P}_{\theta}^2$ over $\theta$ be the collection of all 4-tuples of covectors $(\omega(\theta),\theta^1,\theta^{\overline{1}},\phi)$ on $T_{\theta}({\mathcal P}^1)$, where $\theta^1$ and $\phi$ are chosen as described above. Let $\pi^2: {\mathcal P}^2\ra {\mathcal P}^1$ be the projection. It is easy to see that ${\mathcal P}^2$ is a principal $H$-bundle if considered as a fiber bundle over $M$ with projection $\pi:=\pi^1\circ\pi^2$. Indeed, define (left) actions of  the subgroups $H^1$, $H^2$ on ${\mathcal P}^2$ by the respective formulas
\begin{equation}
\begin{array}{l}
(\theta,\theta^1,\theta^{\overline{1}},\phi)\mapsto (|A|^2\theta,A\theta_*^1,\overline{A}\theta_*^{\overline{1}},\phi_*),\\
\vspace{-0.1cm}\\
(\theta,\theta^1,\theta^{\overline{1}},\phi)\mapsto (\theta,\theta^1+\overline{B}\omega(\theta),\theta^{\overline{1}}-B\omega(\theta),\\
\vspace{-0.5cm}\\
\hspace{6cm}\phi-B\theta^1-\overline{B}\theta^{\overline{1}}-2\Lambda\omega(\theta)),
\end{array}\label{actionbyh}
\end{equation}
where asterisks denote the pushforwards of covectors on $T_{\theta}({\mathcal P}^1)$ to covectors on $T_{|A|^2\theta}({\mathcal P}^1)$ by means of the diffeomorphism of ${\mathcal P}^1$ given by $\theta\mapsto |A|^2\theta$. It is then not hard to verify that formulas (\ref{actionbyh}) yield an action of $H$ on ${\mathcal P}^2$ as required.

We now define two tautological 1-forms on ${\mathcal P}^2$ as
$$
\begin{array}{l}
\omega^1({\bf\Theta})(X):=\theta^1(d\pi^2({\bf\Theta})(X)),\\
\vspace{-0.3cm}\\
\varphi({\bf\Theta})(X):=\phi(d\pi^2({\bf\Theta})(X)),
\end{array}
$$
where ${\bf\Theta}=(\omega(\theta),\theta^1,\theta^{\overline{1}},\phi)$ is a point in ${\mathcal P}_{\theta}^2$ and $X\in T_{{\bf\Theta}}({\mathcal P}^2)$. It is clear from (\ref{integrminus}) that these forms satisfy
\begin{equation}
d\omega=-\omega^1\wedge\omega^{\overline{1}}-\omega\wedge\varphi,\label{integrplushigh}
\end{equation}
where we denote the pull-back of $\omega$ from ${\mathcal P}^1$ to ${\mathcal P}^2$ by the same symbol. Further, computing $d\omega^1$ in local coordinates on ${\mathcal P}^2$ and using the integrability of the CR-structure of $M$ we obtain
\begin{equation}
d\omega^1=\theta^{2}\wedge\xi-\omega^1\wedge\varphi^2-\omega\wedge\varphi^{1}\label{integr1}
\end{equation}
for some complex-valued 1-forms $\theta^2,\xi,\varphi^1,\varphi^2$. Here for any point\linebreak ${\bf\Theta}=(\omega(\theta),\theta^1,\theta^{\overline{1}},\phi)$ the covector $\theta^2({\bf\Theta})$ is the pull-back of a complex-valued covector $\theta_0^2$ at $p:=\pi({\bf\Theta})$ such that $\theta_0^2$ is complex-linear on $H_p(M)$ and the restrictions of $\theta_0^1$ and $\theta_0^2$ to $H_p(M)$ form a basis of $H_p^*(M)$, where $\theta_0^1$ is the covector at $p$ that pulls back to $\theta^1$. 

In \cite{IZ} we study consequences of identities (\ref{integrplushigh}) and (\ref{integr1}). Our calculations are entirely local, and we impose conditions that determine the forms $\theta^2,\varphi^1,\varphi^2$ (as well as another $i\RR$-valued 1-form $\psi$ introduced below) uniquely. This allows us to patch the locally defined forms $\theta^2,\varphi^1,\varphi^2$, $\psi$ into globally defined 1-forms on ${\mathcal P}^2$. Together with $\omega$, $\omega^1$ these globally defined forms are used to construct an absolute ${\mathfrak g}$-valued parallelism on ${\mathcal P}^2$. We will now briefly explain this procedure.

Exterior differentiation of (\ref{integrplushigh}) and substitution of (\ref{integrplushigh}), (\ref{integr1}) for $d\omega$, $d\omega^1$, respectively, yield
\begin{equation}
\begin{array}{l}
(\varphi-\varphi^2-\varphi^{\bar 2})\wedge\omega^1\wedge\omega^{\bar 1}+{\bar\xi}\wedge\omega^1\wedge\theta^{\bar 2}+\xi\wedge\theta^2\wedge\omega^{\bar 1}+\\
\vspace{-0.4cm}\\
\hspace{4cm}(d\varphi-\omega^1\wedge\varphi^{\bar 1}-\omega^{\bar 1}\wedge\varphi^1)\wedge\omega=0.
\end{array}\label{eq1}
\end{equation}
It then follows that
$$
\varphi-\varphi^2-\varphi^{\bar 2}=P\omega^1+\overline{P}\omega^{\bar 1}+Q\theta^2+\overline{Q}\theta^{\bar 2}+R\omega
$$
for some smooth functions $P,Q,R$, where $R$ is $i\RR$-valued. Setting
$$
\tilde\varphi^2:=\varphi^2+P\omega^1+Q\theta^2+\frac{1}{2}R\omega,
$$
we see that the form $\varphi^2$ can be chosen to satisfy identity (\ref{integr1}) with some $\tilde\xi$, $\tilde\varphi^1$ in place of $\xi$, $\varphi^1$ as well as the condition
\begin{equation}
\Re\varphi^2=\frac{\varphi}{2},\label{cond1}
\end{equation}
and from now on we assume that (\ref{cond1}) holds.

With condition (\ref{cond1}) satisfied, identity (\ref{eq1}) implies
$$
\xi=U\theta^2+V\omega^{\bar 1}+W\omega
$$
for some functions $U,V,W$. Setting
$$
\tilde\xi:=\xi-U\theta^2-W\omega,
$$
we therefore can assume that the form $\xi$ is a multiple of $\omega^{\bar 1}$ and satisfies (\ref{integr1}) with some $\tilde\varphi^1$ in place of $\varphi^1$. For $\xi=V\omega^{\bar 1}$ the 2-nondegeneracy of $M$ yields that $V$ is a nowhere vanishing function, thus by scaling $\theta^2$ we suppose from now on that $\xi=\omega^{\bar 1}$. Hence identity (\ref{integr1}) turns into the identity
\begin{equation}
d\omega^1=\theta^{2}\wedge\omega^{\bar 1}-\omega^1\wedge\varphi^2-\omega\wedge\varphi^{1}.\label{integr11}
\end{equation}
Furthermore, for this choice of $\xi$ equation (\ref{eq1}) implies
\begin{equation}
d\varphi=\omega^1\wedge\varphi^{\bar 1}+\omega^{\bar 1}\wedge\varphi^1+2\omega\wedge\psi,\label{integr2}
\end{equation}
where $\psi$ is an  $i\RR$-valued 1-form.

Next, it is not hard to see that the forms $\theta^2$, $\varphi^1$, $\varphi^2$, $\psi$ satisfying (\ref{cond1}), (\ref{integr11}), (\ref{integr2}) are defined up to the following transformations:
\begin{equation}
\begin{array}{l}
\theta^2=\tilde\theta^2+c\omega^1+f\omega,\\
\vspace{-0.3cm}\\
\varphi^2=\tilde\varphi^2-\overline{c}\omega^1+c\omega^{\bar 1}+g\omega,\\
\vspace{-0.3cm}\\
\varphi^1=\tilde\varphi^1+g\omega^1+f\omega^{\bar 1}+r\omega,\\
\vspace{-0.3cm}\\
\displaystyle\psi=\tilde\psi-\frac{\overline{r}}{2}\omega^1+\frac{r}{2}\omega^{\bar 1}+s\omega
\end{array}\label{transform}
\end{equation}
for some functions $c$, $f$, $g$, $r$, $s$, where $g$ and $s$ are $\RR$-valued. We will now impose conditions on $\theta^2$, $\varphi^1$, $\varphi^2$, $\psi$ in order to fix them uniquely.

As computation in local coordinates immediately shows, the values of $i\omega$, $\Re\omega^1$, $\Im\omega^1$, $\Re\theta^2$, $\Im\theta^2$, $\Re\varphi^1$, $\Im\varphi^1$, $\varphi$, $\Im\varphi^2$, $i\psi$ at any ${\bf\Theta}$ constitute a basis of $T_{{\bf\Theta}}^*({\mathcal P}^2)$. In what follows, in order to choose the functions $c$, $f$, $g$, $r$, $s$, we utilize expansions of certain complex-valued forms on ${\mathcal P}^2$ with respect to  $\omega$, $\omega^1$, $\omega^{\bar 1}$, $\theta^2$, $\theta^{\bar 2}$, $\varphi^1$, $\varphi^{\bar 1}$, $\varphi^2$, $\varphi^{\bar 2}$, $\psi$. We will be particularly interested in coefficients at wedge products of $\omega$, $\omega^1$, $\omega^{\bar 1}$, $\theta^2$, $\theta^{\bar 2}$ and for a form $\Omega$ denote them by $\Omega_{\alpha\dots{\bar\beta}\dots\,0}$, where $\alpha,\beta=1,2$, with index 0 corresponding to $\omega$, index 1 to $\omega^1$, and index 2 to $\theta^2$. We will also consider analogous expressions for forms with tildas.

Define
\begin{equation}
\Theta^2:=d\theta^2+\theta^2\wedge(\varphi^2-\varphi^{\bar 2})-\omega^1\wedge\varphi^1\label{Sigma}
\end{equation}
and let $\tilde\Theta^2$ be the 2-form given as in (\ref{Sigma}) by the 1-forms with tildas. In \cite{IZ} we prove
\begin{equation}
\tilde\Theta^2_{2{\bar1}}=\Theta^2_{2{\bar1}}-3c,\label{choiceofc}
\end{equation}
where $\Theta^2_{2{\bar1}}$ and $\tilde\Theta^2_{2{\bar1}}$ are the coefficients at the wedge products $\theta^2\wedge\omega^{\bar 1}$ and $\tilde\theta^2\wedge\omega^{\bar 1}$ in the expansions of $\Theta^2$ and $\tilde\Theta^2$, respectively. This shows that $c$ can be determined by the requirement $\tilde\Theta^2_{2{\bar1}}=0$. Thus, we assume that the condition
\begin{equation}
\Theta^2_{2{\bar1}}=0\label{curv1}
\end{equation}
is satisfied, hence in formula (\ref{transform}) one has $c=0$.

Further, under this assumption in \cite{IZ} we show
\begin{equation}
\Theta^2\equiv
\Theta^2_{1\bar1} \omega^1\wedge\omega^{\bar1}\,\,(\mod\omega, \theta^2\wedge\omega^1)\label{expTheta2first}
\end{equation}
and
\begin{equation}
\tilde\Theta^2_{1{\bar1}}=\Theta^2_{1{\bar1}}+2f.\label{formulaforf}
\end{equation}
Hence, $f$ can be fixed by the requirement $\tilde\Theta^2_{1{\bar1}}=0$, thus we suppose that the condition
\begin{equation}
\Theta^2_{1{\bar1}}=0\label{curv2}
\end{equation}
holds. Therefore, in formula (\ref{transform}) we now have $c=f=0$, i.e.~(\ref{curv1}), (\ref{curv2}) fully determine $\theta^2$. Also, as explained in \cite{IZ}, these conditions yield that the expansion of $\Theta^2$ has the form
\begin{equation}
\Theta^2 = \Theta^2_{21}\theta^2\wedge\omega^1 
+\Theta^2_{20} \theta^2\wedge\omega 
+\Theta^2_{10}\omega^1\wedge\omega
+\Theta^2_{\bar 10} \omega^{\bar 1}\wedge\omega.\label{curvv1}
\end{equation}

Next, to choose the function $g$, we introduce
\begin{equation}
\hspace{-0.3cm}\Phi^2:=d\varphi^2-\theta^2\wedge\theta^{\bar 2}-\omega^1\wedge\varphi^{\bar 1}-\omega\wedge\psi.\label{Phi2}
\end{equation}
Notice that according to (\ref{cond1}), (\ref{integr2}), this form is $i\RR$-valued. In \cite{IZ} we prove
\begin{equation}
\tilde\Phi^2_{1\bar 1}=\Phi^2_{1\bar 1}+2g,\label{formulaforg}
\end{equation}
therefore the real-valued function $g$ can be determined by the requirement $\tilde\Phi^2_{1\bar 1}=0$. Thus, we assume that the condition
\begin{equation}
\Phi^2_{1\bar 1}=0\label{curv3}
\end{equation}
is satisfied. Hence, in formula (\ref{transform}) one has $c=f=g=0$, and we see that (\ref{curv1}), (\ref{curv2}), (\ref{curv3}) completely determine $\varphi^2$.

We will now make a choice of the function $r$. For this purpose we utilize the form
\begin{equation}
\Phi^1:=d\varphi^1+\theta^2\wedge\varphi^{\bar 1}-\omega^1\wedge\psi-\varphi^1\wedge\varphi^{\bar 2}.\label{Phi1}
\end{equation}
In \cite{IZ} we show
\begin{equation}
\tilde\Phi^1_{1\bar 1}=\Phi^1_{1\bar 1}+\frac{3r}{2}.\label{formulaforr}
\end{equation}
Hence $r$ can be fixed by the requirement $\tilde\Phi^1_{1\bar 1}=0$, thus we suppose that the condition
\begin{equation}
\Phi^1_{1\bar 1}=0\label{curv4}
\end{equation}
holds. Hence, in formula (\ref{transform}) one has $c=f=g=r=0$, and one observes that (\ref{curv1}), (\ref{curv2}), (\ref{curv3}), (\ref{curv4}) fully define $\varphi^1$.

Finally, to fix the function $s$, we introduce an $i\RR$-valued form as follows:
\begin{equation}
\Psi:=d\psi+\varphi^1\wedge\varphi^{\bar 1}+\varphi\wedge\psi.\label{Psi}
\end{equation} 
In \cite{IZ} we prove
$$
\tilde\Psi_{1\bar 1}=\Psi_{1\bar 1}+s,
$$
therefore the real-valued function $s$ can be determined by the requirement $\tilde\Psi_{1\bar 1}=0$. Thus, we introduce the condition
\begin{equation}
\Psi_{1\bar 1}=0\label{curv5}
\end{equation}
and observe that (\ref{curv1}), (\ref{curv2}), (\ref{curv3}), (\ref{curv4}), (\ref{curv5}) completely\linebreak define $\psi$.

The forms $\theta^2$, $\varphi^1$, $\varphi^2$, $\psi$ determined by requirements (\ref{curv1}), (\ref{curv2}), (\ref{curv3}), (\ref{curv4}), (\ref{curv5}) give rise to 1-forms on all of ${\mathcal P}^2$, and we denote these globally defined forms by the same respective symbols. Further, let us denote by ${\mathcal P}_M$ the bundle ${\mathcal P}^2$ viewed as a principal $H$-bundle over $M$ and introduce a ${\mathfrak g}$-valued absolute parallelism $\omega_M$ on ${\mathcal P}_M$ by the formula
\begin{equation}
\omega_M:=\left(
\begin{array}{crrrr}
\varphi^2 & \theta^2 & \omega^1&\omega\,& 0\,\,\\
\theta^{\bar 2} & \varphi^{\bar 2} &  \omega^{\bar 1}&0\,\,&-\omega\,\,\\
\varphi^{\bar 1} & \varphi^1 & 0\,\,\,&-\omega^{\bar 1}& -\omega^1\\
\psi & 0\,\, & -\varphi^1&- \varphi^{\bar 2}& -\theta^2\\
0 & -\psi\, & -\varphi^{\bar 1}&-\theta^{\bar 2}& -\varphi^2\\
\end{array}\right).\label{theconnection}
\end{equation}
Notice that in (\ref{theconnection}) we arranged the scalar-valued 1-forms constructed above so that the resulting matrix-valued form indeed takes values in the algebra ${\mathfrak g}$ described in (\ref{liealgebra}). As shown in Theorem 3.2 of \cite{IZ}, the bundle ${\mathcal P}_M$ and the parallelism $\omega_M$ constitute a reduction of the CR-structures in the class ${\mathfrak C}_{2,1}$ to absolute parallelisms.

Next, consider the curvature form $\Omega_M$ of $\omega_M$ given by formula (\ref{genformulacurvature}). Then in terms of matrix elements identities (\ref{integrplushigh}), (\ref{integr11}), (\ref{integr2}) can be written as 
$$
(\Omega_M)^1_4=0,\quad (\Omega_M)^1_3=0,\quad \Re(\Omega_M)^1_1=0,
$$
respectively. Further, using the globally defined forms $\theta^2$, $\varphi^1$, $\varphi^2$, $\psi$ we now introduce the corresponding globally defined 2-forms by formulas (\ref{Sigma}), (\ref{Phi2}), (\ref{Phi1}), (\ref{Psi}) and denote them, as before, by $\Theta^2$, $\Phi^1$, $\Phi^2$, $\Psi$. Then one has
$$
\Theta^2=(\Omega_M)^1_2,\quad\Phi^1=(\Omega_M)^3_2,\quad \Phi^2=(\Omega_M)^1_1,\quad \Psi=(\Omega_M)^4_1.
$$
Below we will be particularly interested in the component $\Theta^2$. Recall that its expansion has the form (\ref{curvv1}).

Let us now go back to the group $G$ and hypersurface $\Gamma$ introduced in Section \ref{model}. Using the (left) action of $G$ on $\Gamma$, we define a right action as
$$
G\times\Gamma\ra\Gamma,\quad (g,p)\mapsto g^{-1}p
$$
and identify $\Gamma$ with the right coset space $H\setminus G$ by means of this right action. Consider the principal $H$-bundle $G\ra H\setminus G\simeq\Gamma$ (with $H$ acting on the total space $G$ by left multiplication) and the right-invariant Maurer-Cartan form $\omega_G^{\hbox{\tiny MC}}$ on $G$. On the other hand, one can associate to $\Gamma$ the bundle ${\mathcal P}_{\Gamma}$ and parallelism $\omega_{\Gamma}$ as constructed in this section. By inspection of our construction one can observe that there exists an isomorphism $F$ of the bundles ${\mathcal P}_{\Gamma}\ra\Gamma$ and $G\ra H\setminus G$ that induces the identity map on the base and such that $F^*\omega_G^{\hbox{\tiny MC}}=\omega_{\Gamma}$. The Maurer-Cartan equation
$$
d\omega_G^{\hbox{\tiny MC}}=\frac{1}{2}\left[\omega_G^{\hbox{\tiny MC}},\omega_G^{\hbox{\tiny MC}}\right]
$$
then implies that the CR-curvature form $\Omega_{\Gamma}$ of $\Gamma$ identically vanishes. Furthermore, as explained in \cite{IZ} (see Corollary 5.1 therein), for\linebreak $M\in{\mathfrak C}_{2,1}$ one has $\Omega_M\equiv 0$ if and only if $M$ is locally CR-equivalent to $\Gamma$, i.e.~for every point $p\in M$ there exists a neighborhood of $p$ that is CR-equivalent to an open subset of $\Gamma$. Thus, CR-flat manifolds in ${\mathfrak C}_{2,1}$ are precisely those that are locally CR-equivalent to $\Gamma$.

We now have all the necessary tools for proving Theorem \ref{main2}. In fact, as we will see in the next section, instead of the assumption of the vanishing of the full curvature tensor, it suffices to require in Theorem \ref{main2} that only the coefficients $\Theta^2_{21}$ and $\Theta^2_{10}$ in expansion of the curvature component $\Theta^2$ are zero. Thus, in the next section we will establish the following stronger fact.

\begin{theorem}\label{main1}
Let $M$ be a tube hypersurface in $\CC^3$ and assume that $M\in{\mathfrak C}_{2,1}$. Suppose further that the coefficients $\Theta^2_{21}$ and $\Theta^2_{10}$ in the expansion of the component $\Theta^2$ of the curvature form $\Omega_M$ vanish identically on ${\mathcal P}_M$. Then $M$ is affinely equivalent to an open subset\linebreak of $M_0$. 
\end{theorem}

As explained in Remark \ref{othercomponents} below, the choice of the curvature coefficients $\Theta^2_{21}$ and $\Theta^2_{10}$ in the statement of Theorem \ref{main1} is in some sense optimal. The proof  given below is based on computing these coefficients in terms of a local defining function of $M$, which requires careful application of the reduction procedure to tube hypersurfaces.

\section{Proof of Theorem \ref{main1}}\label{calculations}
\setcounter{equation}{0}

Let $M$ be any tube hypersurface in $\CC^3$. For $p\in M$, a {\it tube neighborhood}\, of $p$ in $M$ is an open subset $U$ of $M$ that contains $p$ and has the form $M\cap({\mathcal U}+i\RR^3)$, where ${\mathcal U}$ is an open subset of $\RR^3$. It is easy to see that for every point $p\in M$ there exists a tube neighborhood $U$ of $p$ in $M$ and an affine transformation of $\CC^3$ as in (\ref{affequiv}) that maps $p$ to the origin and establishes affine equivalence between $U$ and a tube hypersurface $\Gamma_{\rho}$ of the form
\begin{equation}
z_3+{\bar z}_3=\rho(z_1+{\bar z}_1,z_2+{\bar z}_2),\label{basiceq}
\end{equation}
where $\rho(t_1,t_2)$ is a smooth function defined in a neighborhood of 0 in $\RR^2$ with $\rho(0)=0$, 
$\rho_1(0)=0$, $\rho_2(0)=0$ (here and below subscripts 1 and 2 indicate partial derivatives with respect to $t_1$ and $t_2$). In what follows, $\Gamma_{\rho}$ will be analyzed locally near the origin, hence we will only be interested in the germ of $\rho$ at 0, thus  the domain of $\rho$ will be allowed to shrink if necessary. This convention also applies to all other functions involved in the local analysis of $\Gamma_{\rho}$.

If $M$ is uniformly Levi degenerate of rank 1, then the Hessian matrix of $\rho$ has rank 1 at every point, hence $\rho$ satisfies the homogeneous Monge-Amp\`ere equation
\begin{equation}
\rho_{11}\rho_{22}-\rho_{12}^2\equiv 0,\label{mongeampere}
\end{equation}
and one can additionally assume that $\rho_{11}>0$ everywhere. We will now recall classical facts concerning solutions of equation (\ref{mongeampere}). For details the reader is referred to paper \cite{U}, which treats the equation in somewhat greater generality. 

Let us make the following change of coordinates near the origin
\begin{equation}
\begin{array}{l}
v=\rho_1(t_1,t_2),\\
\vspace{-0.3cm}\\
w=t_2
\end{array}\label{changevar}
\end{equation}
and set
\begin{equation}
\begin{array}{l}
p(v,w):=\rho_2(t_1(v,w),w),\\
\vspace{-0.3cm}\\
q(v):=t_1(v,0).
\end{array}\label{condsfg}
\end{equation}
Equation (\ref{mongeampere}) immediately implies that $p$ is independent of $w$, so we write $p$ as a function of the variable $v$ alone. Furthermore, we have
\begin{equation}
q'(v)=\frac{1}{\rho_{11}(t_1(v,0),0)}.\label{gprime}
\end{equation}
Clearly, (\ref{condsfg}), (\ref{gprime}) yield
\begin{equation}
p(0)=0,\quad q(0)=0,\quad \hbox{$q'>0$ everywhere.}\label{initialconds}
\end{equation}

In terms of $p$ and $q$, change of variables (\ref{changevar}) is inverted as
\begin{equation}
\begin{array}{l}
t_1=q(v)-w\,p'(v),\\
\vspace{-0.3cm}\\
t_2=w,
\end{array}\label{inverttted}
\end{equation}
and the solution $\rho$ in the coordinates $v,w$ is given by
\begin{equation}
\rho(t_1(v,w),w)=vq(v)-\int_{0}^vq(\tau)d\tau+w(p(v)-vp'(v)).\label{solsparam}
\end{equation}
Furthermore, the following holds:
\begin{equation}
\begin{array}{l}
\displaystyle\rho_{11}(t_1(v,w),w)=\displaystyle\frac{1}{q'-w\,p''},\\
\vspace{-0.3cm}\\
\displaystyle\rho_{12}(t_1(v,w),w)=\displaystyle\frac{p'}{q'-w\,p''},\\
\vspace{-0.3cm}\\
\displaystyle\rho_{22}(t_1(v,w),w)=\displaystyle\frac{(p')^2}{q'-w\,p''}.
\end{array}\label{secondpartials}
\end{equation}
In particular, we see that all solutions to the homogeneous Monge-Amp\`ere equation satisfying the conditions 
\begin{equation}
\rho(0)=0,\quad\rho_1(0)=0,\quad \rho_2(0)=0,\quad\hbox{$\rho_{11}>0$ everywhere}\label{grad11}
\end{equation}
are parametrized by a pair of smooth functions satisfying (\ref{initialconds}). 

We now return to our study of the hypersurface $\Gamma_{\rho}$ given by (\ref{basiceq}). Everywhere below, unless stated otherwise, we assume that all functions of the variables $t_1$, $t_2$ are calculated for
\begin{equation}
\begin{array}{l}
t_1=z_1+{\bar z}_1,\\
\vspace{-0.3cm}\\
t_2=z_2+{\bar z}_2.
\end{array}\label{substitution} 
\end{equation}
Using this convention and setting
\begin{equation}
\mu:=\rho_1dz_1+\rho_2dz_2-dz_3\Big |_{\Gamma_{{}_{\rho}}},\quad \eta^1:=\rho_{11}dz_1+\rho_{12}dz_2,\quad \eta^2:=dz_2,\label{aux45}
\end{equation}
we see
\begin{equation}
\begin{array}{l}
\displaystyle d\mu=-\frac{1}{\rho_{11}}\eta^1\wedge\eta^{\bar 1},\\
\vspace{-0.1cm}\\
\displaystyle d\eta^1=-\left(\frac{\rho_{12}}{\rho_{11}}\right)_{\hspace{-0.1cm}1}\eta^2\wedge\eta^{\bar 1}+\left[\frac{\rho_{111}}{\rho_{11}^2}\eta^{\bar 1}+\left(\frac{\rho_{12}}{\rho_{11}}\right)_{\hspace{-0.1cm}1}\eta^{\bar 2}\right]\wedge\eta^1.
\end{array}\label{aux7777}
\end{equation}
To obtain the second equation in (\ref{aux7777}) we use the identity
$$
\rho_{111}\left(\frac{\rho_{12}}{\rho_{11}}\right)^2-2\rho_{112}\left(\frac{\rho_{12}}{\rho_{11}}\right)+\rho_{122}=0,
$$
which is a consequence of (\ref{mongeampere}). Therefore, $\Gamma_{\rho}$ is 2-nondegenerate if and only if the function
\begin{equation}
S:=(\rho_{12}/\rho_{11})_{1}\label{functions}
\end{equation}
vanishes nowhere.

We are now ready to prove Theorem \ref{main1}. We will show that for every $p\in M$ there exists a tube neighborhood $U$ of $p$ in $M$ that is affinely equivalent to an open subset of $M_0$. Observe that this fact implies the statement of the theorem. Indeed, it yields, first of all, that $M$ is real-analytic. Hence, if for some $p\in M$ and a  tube neighborhood $U$ of $p$ we let $F$ be an affine transformation of $\CC^3$ of the form (\ref{affequiv}) with $F(U)\subset M_0$, then the connected real-analytic tube hypersurfaces $F(M)$ and $M_0$ coincide on an open subset. This implies that $F(M)\subset M_0$ since otherwise one would be able to extend $M_0$ to a smooth real-analytic hypersurface in $\CC^3$ containing $i\RR^3$, which is impossible.

Thus, we fix $p\in M$ and find a tube neighborhood $U$ of $p$ as well as a function $\rho$ such that $U$ is affinely equivalent to $\Gamma_{\rho}$ and the following holds: 
$$
\begin{array}{l}
\hbox{(i) $\rho$ satisfies (\ref{mongeampere}), (\ref{grad11}),}\\
\vspace{-0.3cm}\\ 
\hbox{(ii) the function $S$ defined by (\ref{functions}) vanishes nowhere.}
\end{array}
$$
For the hypersurface $\Gamma_{\rho}$ we will now consider the bundles and forms constructed in Section \ref{construction}, and everywhere below all the notation introduced there will be applied to $\Gamma_{\rho}$ in place of $M$. 

For the form $\omega$ on ${\mathcal P}^1$ one has $\omega(u\mu)=u\mu^*$, where $\mu$ is defined in (\ref{aux45}), $u>0$ is the fiber coordinate on ${\mathcal P}^1$ and the asterisk indicates the pull-back of $\mu$ from $\Gamma_{\rho}$ to ${\mathcal P}^1$. Therefore, from the first equation in (\ref{aux7777}) we see
\begin{equation}
d\omega= -\frac{u}{\rho_{11}^*}\eta^{1*}\wedge\eta^{\bar 1*}-\omega\wedge\frac{du}{u}.\label{domeganew}
\end{equation}
Identity (\ref{domeganew}) shows that by setting
\begin{equation}
\nu:=\sqrt{\frac{u}{\rho_{11}^*}}\eta^{1*},\label{aux43}
\end{equation}
one can parametrize the fibers of ${\mathcal P}^2\to{\mathcal P}^1$ as
$$
\begin{array}{l}
\displaystyle\theta^1=a\nu+{\bar b}\omega,\\
\vspace{-0.1cm}\\
\displaystyle\phi=\frac{du}{u}-ab\nu-{\bar a}{\bar b}{\bar \nu}+\lambda\omega,
\end{array}\label{aux459}
$$
with $|a|=1$, $b\in\CC$, $\lambda\in i\RR$ (see (\ref{g1structure})). We then have
\begin{equation}
\begin{array}{l}
\displaystyle\omega^1=a\nu^*+{\bar b}\omega,\\
\vspace{-0.1cm}\\
\displaystyle\varphi=\left(\frac{du}{u}\right)^*-b\omega^1-{\bar b}\omega^{\bar 1}+\lambda\omega,
\end{array}\label{aux41}
\end{equation}
where asterisks indicate pull-backs from ${\mathcal P}^1$ to ${\mathcal P}^2$ and the pull-back of $\omega$ is denoted by the same symbol (cf.~the notation of Section \ref{construction}).

Set
\begin{equation}
\makebox[250pt]{$\begin{array}{l}
\theta^{2(1)}:=-a^2S^{**}\eta^{2**},\\
\vspace{0.1cm}\\
\displaystyle\varphi^{2(1)}:=\frac{da}{a}+\left(\frac{du}{2u}\right)^*+\frac{\bar a^2}{2}\theta^{2(1)}-\frac{a^2}{2}\theta^{\overline{2(1)}}-\\
\vspace{-0.1cm}\\
\displaystyle\hspace{1.2cm}\left(2b+\frac{{\bar a}\rho_{111}^{**}}{2\sqrt{u^*\rho_{11}^{3**}}}\right)\omega^1+\left(\bar b+\frac{a\rho_{111}^{**}}{2\sqrt{u^*\rho_{11}^{3**}}}\right)\omega^{\bar 1}+\frac{\lambda}{2}\omega,\\
\vspace{0.1cm}\\
\displaystyle\varphi^{1(1)}:=d{\bar b}+b\theta^{2(1)}+\overline{b}\varphi^{\overline{2(1)}}-\\
\vspace{-0.1cm}\\
\displaystyle\hspace{1.2cm}\left(2|b|^2-\frac{\lambda}{2}+\frac{ab\rho_{111}^{**}}{2\sqrt{u^*\rho_{11}^{3**}}}+\frac{\bar a\bar b\rho_{111}^{**}}{2\sqrt{u^*\rho_{11}^{3**}}}\right)\omega^1+\bar b^2\omega^{\bar 1}+\frac{\lambda\bar b}{2}\omega,\\
\vspace{0.1cm}\\
\displaystyle \psi^{(1)}:=-\frac{1}{2}d\lambda+\frac{b}{2}\varphi^{1(1)}-\frac{\bar b}{2}\varphi^{\overline{1(1)}}-\frac{\lambda}{2}\varphi+\frac{\lambda b}{4}\omega^1+\frac{\lambda\bar b}{4}\omega^{\bar 1},
\end{array}$}\label{787}
\end{equation}
where double asterisks indicate pull-backs from $\Gamma_{\rho}$ to ${\mathcal P}^2$. We think of these forms as the first approximations of the respective components $\theta^2$, $\varphi^1$, $\varphi^2$, $\psi$ of the parallelism $\omega_{\,\Gamma_{{}_{\rho}}}$ (see (\ref{theconnection})), as indicated by the superscript ${}^{(1)}$. In what follows, we will introduce further approximations of these forms and denote them by the superscripts ${}^{(2)}$, ${}^{(3)}$, etc. We also extend this notation to the curvature form and write $\Theta^{2(k)}$, $\Phi^{1(k)}$, $\Phi^{2(k)}$, $\Psi^{(k)}$ for the components of the curvature form derived from $\theta^{2(k)}$, $\varphi^{1(k)}$, $\varphi^{2(k)}$, $\psi^{(k)}$.  We think of them as the $k$th approximation of the respective components $\Theta^{2}$, $\Phi^{1}$, $\Phi^{2}$, $\Psi$ of the curvature form $\Omega_{\,\Gamma_{{}_{\rho}}}$.

Observe that (\ref{aux41}), (\ref{787}) imply
\begin{equation}
\displaystyle\Re\varphi^{2(1)}=\frac{\varphi}{2},\label{phiphi2}
\end{equation}
which agrees with identity (\ref{cond1}).

Before proceeding further, we record expressions for $(du)^*$, $da$, $db$, $d\lambda$ in terms of the forms $\omega$, $\omega^1$, $\theta^{2(1)}$, $\varphi^{1(1)}$, $\varphi^{2(1)}$, $\psi^{(1)}$ and their conjugates, as they will be used in our future arguments. Indeed, identities (\ref{aux41}), (\ref{787}), (\ref{phiphi2}) yield
\begin{equation}
\makebox[250pt]{$\begin{array}{l}
\displaystyle(du)^*=u^*(\varphi^{2(1)}+\varphi^{\overline{2(1)}}+b\omega^1+{\bar b}\omega^{\bar 1}-\lambda\omega),\\
\vspace{-0.1cm}\\
\displaystyle da=a\left[-\frac{\bar a^2}{2}\theta^{2(1)}+\frac{a^2}{2}\theta^{\overline{2(1)}}+\frac{1}{2}\varphi^{2(1)}-\frac{1}{2}\varphi^{\overline{2(1)}}+\right.\\
\vspace{-0.3cm}\\
\displaystyle\hspace{2cm}\left.\left(\frac{3b}{2}+\frac{{\bar a}\rho_{111}^{**}}{2\sqrt{u^*\rho_{11}^{3**}}}\right)\omega^1-\left(\frac{3\bar b}{2}+\frac{a\rho_{111}^{**}}{2\sqrt{u^*\rho_{11}^{3**}}}\right)\omega^{\bar 1}\right],\\
\vspace{-0.1cm}\\
\displaystyle db=-\overline{b}\theta^{\overline{2(1)}}+\varphi^{\overline{1(1)}}-b\varphi^{2(1)}-b^2\omega^1+\\
\vspace{-0.3cm}\\
\displaystyle \hspace{2cm}\left(2|b|^2+\frac{\lambda}{2}+\frac{ab\rho_{111}^{**}}{2\sqrt{u^*\rho_{11}^{3**}}}+\frac{\bar a\bar b\rho_{111}^{**}}{2\sqrt{u^*\rho_{11}^{3**}}}\right)\omega^{\bar 1}-\frac{\lambda b}{2}\omega,\\
\vspace{-0.1cm}\\
\displaystyle d\lambda=b\varphi^{1(1)}-\bar b\varphi^{\overline{1(1)}}-\lambda\left(\varphi^{2(1)}+\varphi^{\overline{2(1)}}\right)-2\psi^{(1)}+\frac{\lambda b}{2}\omega^1+\frac{\lambda\bar b}{2}\omega^{\bar 1}.
\end{array}$}\label{differentials}
\end{equation}  

In the future, it will be sufficient to perform some of the calculations only on the section $\gamma_0$ of ${\mathcal P}_{\Gamma_{{}_\rho}}\to {\Gamma}_{\rho}$ given by
\begin{equation}
\gamma_0:\quad u^*=1,\quad a=1,\quad b=0,\quad \lambda=0.\label{sectiongamma0}
\end{equation}
On this section formulas (\ref{differentials}) simplify as
\begin{equation}
\begin{array}{l}
\displaystyle(du)^*=\varphi^{2(1)}+\varphi^{\overline{2(1)}},\\
\vspace{-0.1cm}\\
\displaystyle da=-\frac{1}{2}\theta^{2(1)}+\frac{1}{2}\theta^{\overline{2(1)}}+\frac{1}{2}\varphi^{2(1)}-\frac{1}{2}\varphi^{\overline{2(1)}}+\\
\vspace{-0.3cm}\\
\displaystyle\hspace{6cm}\frac{\rho_{111}^{**}}{2\sqrt{\rho_{11}^{3**}}}\omega^1-\frac{\rho_{111}^{**}}{2\sqrt{\rho_{11}^{3**}}}\omega^{\bar 1},\\
\vspace{-0.3cm}\\
\displaystyle db=\varphi^{\overline{1(1)}},\quad d\lambda=-2\psi^{(1)}.\\
\end{array}\label{differentialsat0}
\end{equation}

Next, by a somewhat lengthy computation utilizing formulas (\ref{integrplushigh}), (\ref{aux45}), (\ref{aux7777}), (\ref{functions}), (\ref{aux43}), (\ref{aux41}), (\ref{787}), (\ref{differentials}), one verifies that the following holds: 
\begin{equation}
\begin{array}{l}
\displaystyle d\omega^1=\theta^{2(1)}\wedge\omega^{\bar 1}-\omega^1\wedge\varphi^{2(1)}-\omega\wedge\varphi^{1(1)},\\
\vspace{-0.1cm}\\
d\varphi=\omega^1\wedge\varphi^{\overline{1(1)}}+\omega^{\bar 1}\wedge\varphi^{1(1)}+2\omega\wedge\psi^{(1)},
\end{array}\label{aux111}
\end{equation}  
which agrees with identities (\ref{integr11}), (\ref{integr2}). Recall that (\ref{cond1}), (\ref{integr11}), (\ref{integr2}) lie at the foundation of our construction in Section \ref{construction}.

We will now calculate the coefficient $\Theta^2_{21}$ in the expansion of the component $\Theta^2$ of the curvature form of $\Gamma_{\rho}$. In fact, for our purposes it will be sufficient to find this coefficient only on the section $\gamma_0$ (see (\ref{sectiongamma0})). In order to do this, we introduce second approximations of the forms $\theta^2$, $\varphi^1$, $\varphi^2$, $\psi$. Namely, according to the procedure described in Section \ref{construction}, we set (cf.~formula (\ref{transform})):
\begin{equation}
\begin{array}{l}
\theta^{2(2)}:=\theta^{2(1)}-c\omega^1,\\
\vspace{-0.1cm}\\
\varphi^{2(2)}:=\varphi^{2(1)}+{\bar c}\omega^1-c\omega^{\bar 1},\\
\vspace{-0.1cm}\\
\varphi^{1(2)}:=\varphi^{1(1)},\\
\vspace{-0.1cm}\\
\psi^{(2)}:=\psi^{(1)},
\end{array}\label{secondapprox}
\end{equation} 
where $c$ is chosen so that the expansion of $\Theta^{2(2)}$ (equivalently, that of $d\theta^{2(2)}$) does not involve $\theta^{2(2)}\wedge\omega^{\bar 1}$ (see (\ref{Sigma})). By formula (\ref{choiceofc}), the function $c$ is given by
\begin{equation}
c=\frac{1}{3}\Theta^{2(1)}_{2\bar 1}.\label{paramc}
\end{equation}
Observe that the coefficient $\Theta^{2(1)}_{2\bar 1}$ in the above formula is equal to that at the wedge product $\theta^{2(1)}\wedge\omega^{\bar 1}$ in the expansion of $d\theta^{2(1)}$. Differentiating the first equation in (\ref{787}) and using (\ref{aux45}), (\ref{aux43}), (\ref{aux41}), (\ref{787}), (\ref{differentials}), we then obtain
\begin{equation} 
\Theta^{2(1)}_{2\bar 1}=-\frac{aS_{1}^{**}}{\sqrt{u^*\rho_{11}^{**}}S^{**}}+\frac{a\rho_{111}^{**}}{\sqrt{u^*\rho_{11}^{3**}}}+{3\bar b}.\label{aux87}
\end{equation} 

Notice now that one has $\Theta^{2}_{21}=\Theta^{2(2)}_{21}$ since, by formulas (\ref{Sigma}) and (\ref{expTheta2first}), transformations of the form (\ref{transform}) with $c=0$ cannot change the value of this coefficient. Clearly, $\Theta^{2(2)}_{21}$ is equal to the coefficient at the wedge product $\theta^{2(2)}\wedge\omega^1$ in the expansion of $d\theta^{2(2)}$. Thus, we need to differentiate the first identity in (\ref{secondapprox}), which involves differentiating the function $c$. These calculations are quite substantial, but they significantly simplify when restricted to the section $\gamma_0$ (see (\ref{sectiongamma0})). Then, utilizing formulas (\ref{aux45}), (\ref{aux43}), (\ref{aux41}), (\ref{787}), (\ref{differentialsat0}), (\ref{aux111}), (\ref{secondapprox}), (\ref{paramc}), (\ref{aux87}), we arrive at the following result:
\begin{equation}
\makebox[250pt]{$\begin{array}{l}
\displaystyle\Theta^{2(2)}_{21}\Big|_{\gamma_0}=\frac{1}{3S^{**}}\left[\frac{\rho_{12}^{**}}{\rho_{11}^{**}}\left(\frac{S_{1}^{**}}{\sqrt{\rho_{11}^{**}}S^{**}}\right)_{\hspace{-0.1cm}1}-\left(\frac{S_{1}^{**}}{\sqrt{\rho_{11}^{**}}S^{**}}\right)_{\hspace{-0.1cm}2}\right]-\\
\vspace{-0.1cm}\\
\displaystyle\hspace{0.4cm} \frac{1}{3S^{**}}\left[\frac{\rho_{12}^{**}}{\rho_{1 1}^{**}}\left(\frac{\rho_{111}^{**}}{\sqrt{\rho_{11}^{3**}}}\right)_{\hspace{-0.1cm}1}-\left(\frac{\rho_{111}^{**}}{\sqrt{\rho_{11}^{3**}}}\right)_{\hspace{-0.1cm}2}\right]-\frac{11S_{1}^{**}}{6\sqrt{\rho_{11}^{**}}\,\,{S^{**}}}-\frac{\rho_{111}^{**}}{6\sqrt{\rho_{11}^{3**}}},
\end{array}$}\label{veryfinaltheta}
\end{equation}
where double asterisks indicate pull-backs from $\Gamma_{\rho}$ to $\gamma_0$.

We will now use the assumption $\Theta^{2}_{21}\equiv 0$ of the theorem. From formula (\ref{veryfinaltheta}) one obtains
\begin{equation}
\makebox[250pt]{$\begin{array}{l}
\displaystyle2{\sqrt{\rho_{11}}}\left[\rho_{12}\left(\frac{S_{1}}{\sqrt{\rho_{11}}S}\right)_{\hspace{-0.1cm}1}-\rho_{11}\left(\frac{S_{1}}{\sqrt{\rho_{11}}S}\right)_{\hspace{-0.1cm}2}\right]-\\
\vspace{-0.1cm}\\
\displaystyle\hspace{0.4cm}2{\sqrt{\rho_{11}}}\left[\rho_{12}\left(\frac{\rho_{111}}{\sqrt{\rho_{11}^3}}\right)_{\hspace{-0.1cm}1}-\rho_{11}\left(\frac{\rho_{111}}{\sqrt{\rho_{11}^3}}\right)_{\hspace{-0.1cm}2}\right]-{11S_{1}}\,\rho_{11}-{S\,\rho_{111}}\equiv 0.
\end{array}$}\label{veryfinalthetav}
\end{equation}
Note that in (\ref{veryfinalthetav}) we dropped asterisks and no longer need to assume that substitution (\ref{substitution}) takes place, thus the left-hand side of (\ref{veryfinalthetav}) is regarded as a function on a neighborhood of the origin in $\RR^2$.

Further, formulas (\ref{secondpartials}) can be used to rewrite (\ref{veryfinalthetav}) in the coordinates $v$, $w$ introduced in (\ref{changevar}). Namely, one has
\begin{equation}
\begin{array}{l}
\displaystyle S(t_1(v,w),w)=\frac{p''}{q'-w\,p''},\\
\vspace{-0.1cm}\\
\displaystyle S_1(t_1(v,w),w)=\frac{p'''q'-p''q''}{(q'-w\,p'')^3},\\
\vspace{-0.1cm}\\
\displaystyle \rho_{111}(t_1(v,w),w)=-\frac{q''-w\,p'''}{(q'-w\,p'')^3},
\end{array}\label{ids444}
\end{equation}
and then it is not hard to see that (\ref{veryfinalthetav}) is equivalent to
\begin{equation}
p'''q'-p''q''\equiv 0,\label{vanish1}
\end{equation}
that is, to the condition $S_1\equiv 0$. Since $S$ vanishes nowhere, the first identity in (\ref{ids444}) implies that $p''$ does not vanish either. Then, dividing (\ref{vanish1}) by $(p'')^2$, one obtains $q'/p''\equiv\hbox{const}$, which yields, upon taking into account conditions (\ref{initialconds}), the identity
\begin{equation}
q=C(p'-D),\label{exprg}
\end{equation}
where $D:=p'(0)$ and $C$ is a constant satisfying $C\,p''>0$.

Thus, the assumption $\Theta^{2}_{21}\equiv 0$ leads to relation (\ref{exprg}) between $p$ and $q$. We will now interpret this relation in terms of the function $\rho$. Let $\zeta$ be the the inverse of the function $D-p'$ near the origin. Define
\begin{equation}
\chi(\tau):=\frac{1}{\tau}\int_{0}^{\tau}\zeta(\sigma)d\sigma.\label{functionchi}
\end{equation}
Clearly, $\chi$ is smooth near 0 and satisfies
\begin{equation}
\chi(0)=0,\quad \chi'(0)=-\frac{1}{2p''(0)}.\label{chiconds}
\end{equation} 
Now set
\begin{equation}
\tilde\rho(t_1,t_2):=(t_1+ D t_2)\chi\left(\frac{t_1+ D t_2}{t_2-C}\right).\label{deftilderho}
\end{equation}

\begin{lemma}\label{firstcondrho} \it One has $\rho=\tilde\rho$.
\end{lemma}

\begin{proof}
From (\ref{deftilderho}) we compute:
\begin{equation}
\makebox[250pt]{$\begin{array}{l}
\displaystyle\tilde\rho_1=\chi\left(\frac{t_1+ D t_2}{t_2-C}\right)+\frac{t_1+ D t_2}{t_2-C}\chi'\left(\frac{t_1+ D t_2}{t_2-C}\right),\\
\vspace{-0.1cm}\\
\displaystyle\tilde\rho_2=D\,\chi\left(\frac{t_1+ D t_2}{t_2-C}\right)+\\
\vspace{-0.4cm}\\
\displaystyle\hspace{2cm}(t_1+ D t_2)\left(\frac{D}{t_2-C}-\frac{t_1+Dt_2}{(t_2-C)^2}\right)\chi'\left(\frac{t_1+ D t_2}{t_2-C}\right),\\
\vspace{-0.1cm}\\
\displaystyle\tilde\rho_{11}=\frac{2}{t_2-C}\chi'\left(\frac{t_1+ D t_2}{t_2-C}\right)+\frac{t_1+ D t_2}{(t_2-C)^2}\chi''\left(\frac{t_1+ D t_2}{t_2-C}\right).\\
\end{array}$}\label{tilderhoderiv}
\end{equation}
Using (\ref{chiconds}), (\ref{deftilderho}) we then see
$$
\tilde\rho(0)=0,\quad\tilde\rho_1(0)=0,\quad \tilde\rho_2(0)=0,\quad \tilde\rho_{11}>0,
$$
and it is easy to observe that $\tilde\rho$ satisfies Monge-Amp\`ere equation (\ref{mongeampere}). Hence, $\tilde\rho$ is fully determined by a pair of functions $\tilde p$, $\tilde q$ as in formulas (\ref{inverttted}),  (\ref{solsparam}). These functions satisfy
$$
\tilde p(0)=0,\quad \tilde q(0)=0,\quad \hbox{$\tilde q\,'>0$ everywhere}
$$
(cf.~conditions (\ref{initialconds})).

Let us make a change of coordinates near the origin analogous to change (\ref{changevar}):
\begin{equation}
\begin{array}{l}
\tilde v=\tilde\rho_1(t_1,t_2),\\
\vspace{-0.3cm}\\
\tilde w=t_2.
\end{array}\label{changevar1}
\end{equation}
Then (\ref{tilderhoderiv}) yields
$$
\tilde v=(D-p')^{-1}\left(\frac{t_1+ D t_2}{t_2-C}\right)
$$
and therefore, taking into account (\ref{exprg}), we see that (\ref{changevar1}) is inverted as
$$
\begin{array}{l}
t_1=C(p'(\tilde v)-D)-\tilde wp'(\tilde v)=q(\tilde v)-\tilde w p'(\tilde v),\\
\vspace{-0.3cm}\\
t_2=\tilde w.
\end{array}
$$
On the other hand, we have, as in (\ref{inverttted})
$$
t_1=\tilde q(\tilde v)-\tilde w \tilde p\,'(\tilde v).
$$
Hence, it follows that $\tilde q=q$ and, since $\tilde p(0)=p(0)=0$, one also has $\tilde p=p$. Therefore, $\tilde\rho=\rho$,  and the proof is complete. \end{proof}

Next, we calculate the coefficient $\Theta^2_{10}$ in the expansion of the component $\Theta^2$ of the curvature form of $\Gamma_{\rho}$. In fact, it will be sufficient for our purposes to determine $\Theta^2_{10}$ only on the section $\gamma_0$ of ${\mathcal P}_{{\Gamma}_{{}_{\rho}}}\ra\Gamma_{\rho}$ (see (\ref{sectiongamma0})). This is computationally much harder to do than finding $\Theta^{2}_{21}|_{\gamma_0}$ and will require introducing three additional approximations of the forms $\theta^2$, $\varphi^1$, $\varphi^2$, $\psi$. We start with the third approximation:
\begin{equation}
\begin{array}{l}
\theta^{2(3)}:=\theta^{2(2)}-f\omega,\\
\vspace{-0.1cm}\\
\varphi^{2(3)}:=\varphi^{2(2)},\\
\vspace{-0.1cm}\\
\varphi^{1(3)}:=\varphi^{1(2)}-f\omega^{\bar 1},\\
\vspace{-0.1cm}\\
\psi^{(3)}:=\psi^{(2)}.
\end{array}\label{thirdapprox}
\end{equation} 
Here $f$ is chosen so that the expansion of $\Theta^{2(3)}$ (equivalently, that of $d\theta^{2(3)}$) does not involve $\omega^1\wedge\omega^{\bar 1}$. By formula (\ref{formulaforf}), the function $f$ is given by
\begin{equation}
f=-\frac{1}{2}\Theta^{2(2)}_{1\bar 1},\label{paramf}
\end{equation}
where the coefficient $\Theta^{2(2)}_{1\bar 1}$ in fact coincides with that at the wedge product $\omega^1\wedge\omega^{\bar 1}$ in the expansion of $d\theta^{2(2)}$. Therefore, we need to differentiate the first equation in (\ref{secondapprox}), which leads to rather lengthy calculations. They can be accomplished with the help of formulas (\ref{aux45}), (\ref{aux43}), (\ref{aux41}), (\ref{787}), (\ref{differentials}),  (\ref{aux111}), (\ref{secondapprox}), (\ref{paramc}), (\ref{aux87}), and one obtains
\begin{equation} 
\Theta^{2(2)}_{1\bar 1}=\frac{a^2\rho^{\rm{(IV)}**}}{3u^*\rho_{11}^{2**}}+\frac{2a\bar b\rho^{**}_{111}}{3\sqrt{u^*\rho_{11}^{3**}}}-\frac{4a^2\rho_{111}^{2**}}{9u^*\rho_{11}^{3**}}+\bar b^2,
\label{aux87f}
\end{equation}
where $\rho^{{\rm(IV)}}:=\partial^{\,4}\rho/\partial\,t_1^4$. 

Next, we introduce the fourth approximation of the forms $\theta^2$, $\varphi^1$, $\varphi^2$, $\psi$ as
\begin{equation}
\begin{array}{l}
\theta^{2(4)}:=\theta^{2(3)},\\
\vspace{-0.1cm}\\
\varphi^{2(4)}:=\varphi^{2(3)}-g\omega,\\
\vspace{-0.1cm}\\
\varphi^{1(4)}:=\varphi^{1(3)}-g\omega^1,\\
\vspace{-0.1cm}\\
\psi^{(4)}:=\psi^{(3)},
\end{array}\label{fourthapprox}
\end{equation} 
with $g$ chosen so that the expansion of $\Phi^{2(4)}$ (equivalently, that of $d\varphi^{2(4)}$) does not involve $\omega^1\wedge\omega^{\bar 1}$ (see (\ref{Phi2})). By (\ref{formulaforg}), the function $g$ is given by
\begin{equation}
g=-\frac{1}{2}\Phi^{2(3)}_{1\bar 1},\label{paramg}
\end{equation}
with the coefficient $\Phi^{2(3)}_{1\bar 1}$ equal to that at the wedge product $\omega^1\wedge\omega^{\bar 1}$ in the expansion of $d\varphi^{2(3)}$ (this form coincides with $d\varphi^{2(2)}$). Thus, we need to differentiate the second equation in (\ref{secondapprox}), which again requires substantial calculations. They can be performed by utilizing formulas (\ref{aux45}), (\ref{aux43}), (\ref{aux41}), (\ref{787}), (\ref{differentials}), (\ref{aux111}), (\ref{secondapprox}), (\ref{paramc}), (\ref{aux87}), (\ref{thirdapprox}), (\ref{paramf}), (\ref{aux87f}), and one arrives at the following expression:
\begin{equation} 
\Phi^{2(3)}_{1\bar 1}=\frac{\rho^{{\rm(IV)}**}}{3u^*\rho_{11}^{2**}}+\frac{a b\rho^{**}_{111}}{3\sqrt{u^*\rho_{11}^{3**}}}+\frac{\bar a \bar b\rho^{**}_{111}}{3\sqrt{u^*\rho_{11}^{3**}}}-\frac{4\rho_{111}^{2**}}{9u^*\rho_{11}^{3**}}+|b|^2.
\label{aux87g}
\end{equation} 

Finally, we define the fourth approximation of the forms $\theta^2$, $\varphi^1$, $\varphi^2$, $\psi$ by
\begin{equation}
\begin{array}{l}
\theta^{2(5)}:=\theta^{2(4)},\\
\vspace{-0.1cm}\\
\varphi^{2(5)}:=\varphi^{2(4)},\\
\vspace{-0.1cm}\\
\varphi^{1(5)}:=\varphi^{1(4)}-r\omega,\\
\vspace{-0.1cm}\\
\displaystyle\psi^{(5)}:=\psi^{(4)}+\frac{\bar r}{2}\omega^1-\frac{r}{2}\omega^{\bar 1},
\end{array}\label{fifthapprox}
\end{equation} 
where $r$ is chosen so that the expansion of $\Phi^{1(5)}$ (equivalently, that of $d\varphi^{1(5)}$) does not involve $\omega^1\wedge\omega^{\bar 1}$ (see (\ref{Phi1})). By (\ref{formulaforr}), the function $r$ is given by
\begin{equation}
r=-\frac{2}{3}\Phi^{1(4)}_{1\bar 1},\label{paramr}
\end{equation}
with the coefficient $\Phi^{1(4)}_{1\bar 1}$ equal to that at the wedge product $\omega^1\wedge\omega^{\bar 1}$ in the expansion of $d\varphi^{1(4)}$. Therefore, we have to differentiate the third equation in (\ref{fourthapprox}), which involves extensive calculations. They are accomplished with the help of formulas (\ref{aux45}), (\ref{aux43}), (\ref{aux41}), (\ref{787}), (\ref{differentials}), (\ref{aux111}), (\ref{secondapprox}), (\ref{paramc}), (\ref{aux87}), (\ref{thirdapprox}), (\ref{paramf}), (\ref{aux87f}), (\ref{fourthapprox}), (\ref{paramg}), (\ref{aux87g}), and we obtain
\begin{equation}
\makebox[250pt]{$\begin{array}{l}
\displaystyle\Phi^{1(4)}_{1\bar 1}=\frac{(4a^2b-ab-\bar a\bar b-2\bar b)\rho^{{\rm(IV)}**}}{6u^*\rho_{11}^{2**}}-\\
\vspace{-0.3cm}\\
\hspace{4cm}\displaystyle\frac{(11a^2b-3ab-3\bar a\bar b-5\bar b)\rho_{111}^{2**}}{12u^*\rho_{11}^{3**}}+\\
\vspace{-0.3cm}\\
\hspace{7.5cm}\displaystyle\frac{(a|b|^2-\bar a\bar b^2)\rho_{111}^{**}}{4\sqrt{u^*\rho_{11}^{3**}}}-\frac{3\lambda\bar b}{2}.
\end{array}$}
\label{aux87r}
\end{equation} 

Observe now that one has $\Theta^{2}_{10}=\Theta^{2(5)}_{10}$ since, by formulas (\ref{Sigma}), (\ref{curvv1}), transformations of the form (\ref{transform}) with $c=f=g=r=0$ cannot change the value of this coefficient. Clearly, $\Theta^{2(5)}_{10}$ is equal to the coefficient at the wedge product $\omega^1\wedge\omega$ in the expansion of $d\theta^{2(5)}$, which coincides with $d\theta^{2(3)}$. Thus, we need to differentiate the first identity in (\ref{thirdapprox}). These calculations are lengthy, but they simplify substantially when restricted to the section $\gamma_0$ (see (\ref{sectiongamma0})). Then, utilizing formulas (\ref{aux45}), (\ref{aux43}), (\ref{aux41}), (\ref{787}), (\ref{differentialsat0}), (\ref{aux111}), (\ref{secondapprox}), (\ref{paramc}), (\ref{aux87}), (\ref{thirdapprox}), (\ref{paramf}), (\ref{aux87f}), (\ref{fourthapprox}), (\ref{paramg}), (\ref{aux87g}), (\ref{fifthapprox}), (\ref{paramr}), (\ref{aux87r}), we arrive at the following result:
\begin{equation}
\displaystyle\Theta^{2(5)}_{10}\Big|_{\gamma_0}=\frac{\rho^{{\rm(V)}**}}{6\sqrt{\rho_{11}^{5**}}}-\frac{5\rho^{{\rm(IV)}**}\rho_{111}^{**}}{6\sqrt{\rho_{11}^{7**}}}+\frac{20\rho_{111}^{3**}}{27\sqrt{\rho_{11}^{9**}}},
\label{veryfinalthetass}
\end{equation}
where $\rho^{{\rm(V)}}:=\partial^{\,5}\rho/\partial\,t_1^5$ and double asterisks indicate pull-backs from $\Gamma_{\rho}$ to $\gamma_0$.

We will now use the assumption $\Theta^{2}_{10}\equiv 0$ of the theorem. From formula (\ref{veryfinalthetass}) one obtains
\begin{equation}
9\rho^{{\rm(V)}}\rho_{11}^{2}-45\rho^{{\rm(IV)}}\rho_{111}\rho_{11}+40\rho_{111}^{3}\equiv 0.
\label{veryfinalthetasss}
\end{equation}
Notice that in (\ref{veryfinalthetasss}) we dropped asterisks and no longer assume that substitution (\ref{substitution}) takes place, thus the left-hand side of (\ref{veryfinalthetasss}) is regarded as a function near the origin in $\RR^2$.

Further, formulas (\ref{secondpartials}), (\ref{ids444}) can be used to rewrite (\ref{veryfinalthetasss}) in the coordinates $v$, $w$ introduced in (\ref{changevar}). Indeed, one has
$$
\makebox[250pt]{$\begin{array}{l}
\displaystyle \rho^{{\rm(IV)}}(t_1(v,w),w)=-\frac{1}{(q'-w\,p'')^5}\Bigl[(q'''-w\,p^{{\rm(IV)}})(q'-w\,p'')-\\
\vspace{-0.6cm}\\
\displaystyle\hspace{9cm}3(q''-w\,p''')^2\Bigr],\\
\vspace{-0.1cm}\\
\displaystyle \rho^{{\rm(V)}}(t_1(v,w),w)=-\frac{1}{(q'-w\,p'')^7}\Bigl[\Bigl((q^{{\rm(IV)}}-w\,p^{{\rm(V)}})(q'-w\,p'')+\\
\vspace{-0.4cm}\\
\displaystyle\hspace{1cm}(q'''-w\,p^{{\rm(IV)}})(q''-w\,p''')-6(q''-w\,p''')(q'''-w\,p^{{\rm(IV)}})\Bigr)\times\\
\vspace{-0.4cm}\\
\displaystyle\hspace{1cm}(q'-w\,p'')-5\Bigl((q'''-w\,p^{{\rm(IV)}})(q'-w\,p'')-3(q''-w\,p''')^2\Bigr)\times\\
\vspace{-0.4cm}\\
\displaystyle\hspace{9.3cm}(q''-w\,p''')\Bigr].
\end{array}$}
$$
Then, after some calculations, identity (\ref{veryfinalthetasss}) reduces to the equation
\begin{equation}
9p^{{\rm(V)}}(p'')^2-45p^{{\rm(IV)}}p'''p''+40(p''')^3\equiv 0,\label{final1}
\end{equation}
where we took into account relation (\ref{exprg}).

Recalling that $Cp''>0$, for the left-hand side of (\ref{final1}) we observe
\begin{equation}
\begin{array}{l}
\displaystyle\frac{C}{(Cp'')^{11/3}}\Bigl(9p^{{\rm(V)}}(p'')^2-45p^{{\rm(IV)}}p'''p''+40(p''')^3\Bigr)=\\
\vspace{-0.5cm}\\
\displaystyle\hspace{5cm}\left(\frac{9p^{{\rm(IV)}}}{C(Cp'')^{5/3}}-\frac{15(p''')^2}{(Cp'')^{8/3}}\right)'.\label{final2}
\end{array}
\end{equation}
Further, notice that
\begin{equation}
\frac{C}{9}\left(\frac{9p^{{\rm(IV)}}}{C(Cp'')^{5/3}}-\frac{15(p''')^2}{(Cp'')^{8/3}}\right)=\left(\frac{p'''}{(Cp'')^{5/3}}\right)'\label{final3}
\end{equation}
and that
\begin{equation}
-\frac{2C}{3}\frac{p'''}{(Cp'')^{5/3}}=\left(\frac{1}{(Cp'')^{2/3}}\right)'.\label{final4}
\end{equation}
Now, identities (\ref{final1}), (\ref{final2}), (\ref{final3}), (\ref{final4}) imply
$$
p''(v)=\frac{1}{C(C_1v^2+C_2v+C_3)^{3/2}}
$$
for some constants $C_1, C_2, C_3$ with $C_3>0$. 

We now let $\Delta:=C_2^2-4C_1C_3$ and consider three cases. The formulas that appear below contain constants of integration as well as complicated expressions in $C$, $C_1$, $C_2$, $C_3$. We write all these constants as $C_4, C_5$, etc.

\noindent {\bf Case 1:} $C_1=C_2=0$. In this situation we have
$$
p'(v)=\frac{v}{C\, C_3^{3/2}}+C_4.
$$
Therefore, the function $\chi$ defined in (\ref{functionchi}) is linear, and Lemma \ref{firstcondrho} yields 
$$
\rho(t_1,t_2)=C_5\frac{(t_1+Dt_2)^2}{t_2-C},
$$
where $C_5\ne 0$. It then follows that the hypersurface $\Gamma_{\rho}$ is affinely equivalent to an open subset of the tube hypersurface with the base given by
\begin{equation}
x_1x_2=x_3^2,\quad x_1>0.\label{hypfinal1}
\end{equation}
Clearly, (\ref{hypfinal1}) is affinely equivalent to an open subset of $M_0$.

\noindent {\bf Case 2:} $\Delta=0$, $C_2\ne 0$ (hence $C_1>0$). In this situation we find
$$
p'(v)=-\frac{2\sqrt{C_1}}{C(2C_1v+C_2)^2}+C_4.
$$
After some calculations this yields
$$
\chi(\tau)=\frac{C_5}{\tau}\Bigl(\sqrt{C_6\tau+2C_1^{1/2}}-\sqrt{2}C_1^{1/4}\Bigr)-\frac{C_5C_6}{2^{3/2}C_1^{1/4}},
$$
where $C_5\ne 0$, $C_6\ne 0$. Then by Lemma \ref{firstcondrho} we obtain
$$
\begin{array}{l}
\displaystyle\rho(t_1,t_2)=C_5(t_2-D)\sqrt{C_6\frac{t_1+Dt_2}{t_2-D}+2C^{1/2}}+\\
\vspace{-0.3cm}\\
\displaystyle\hspace{5cm}C_7(t_1+Dt_2)+C_8(t_2-D).
\end{array}
$$
Again, it immediately follows that $\Gamma_{\rho}$ is affinely equivalent to an open subset of the tube hypersurface with base (\ref{hypfinal1}), hence to an open subset of $M_0$.

\noindent {\bf Case 3:} $\Delta\ne 0$. In this situation we compute
\begin{equation}
p'(v)=-\frac{4C_1v+2C_2}{C\,\Delta\,\sqrt{C_1v^2+C_2v+C_3}}+C_4.\label{pprimecase3}
\end{equation}

If $C_1=0$, (\ref{pprimecase3}) yields
$$
\chi(\tau)=\frac{C_5}{C_6\tau+1}-C_5,
$$
where $C_5\ne 0$, $C_6\ne 0$. Hence, by Lemma \ref{firstcondrho} we obtain
$$
\displaystyle\rho(t_1,t_2)=\frac{C_5(t_1+Dt_2)(t_2-C)}{C_6(t_1+Dt_2)+(t_2-C)}+C_7(t_1+Dt_2).
$$
It then follows that $\Gamma_{\rho}$ is affinely equivalent to an open subset of the tube hypersurface with the base given by
$$
x_1x_2=x_3(x_1+x_2),\quad x_1>0.
$$
It is easy to see that this hypersurface is affinely equivalent to an open subset of $M_0$.

Further, if $C_1\ne 0$, (\ref{pprimecase3}) implies
$$
\chi(\tau)=\frac{C_5}{\tau}\Bigl(\sqrt{C_6(\tau+C_7)^2+C_8}-\sqrt{C_6C_7^2+C_8}\Bigr)-\frac{C_5C_6C_7}{\sqrt{C_6C_7^2+C_8}},
$$
where $C_5\ne 0$, $C_6\ne 0$, $C_8\ne 0$. Therefore, Lemma \ref{firstcondrho} yields
$$
\begin{array}{l}
\displaystyle\rho(t_1,t_2)=C_5(t_2-D)\sqrt{C_6\left(\frac{t_1+Dt_2}{t_2-D}+C_7\right)^2+C_8}+\\
\vspace{-0.4cm}\\
\displaystyle\hspace{7cm}C_9(t_1+Dt_2)+C_{10}(t_2-D).
\end{array}
$$
Once again, it is not hard to see that $\Gamma_{\rho}$ is affinely equivalent to an open subset of $M_0$.

The proof of Theorem \ref{main1} is now complete.\qed

\begin{remark}\label{othercomponents} Arguing as in the proof of Theorem \ref{main1} above, one can in fact calculate the full curvature form $\Omega_{\Gamma_{{}_{\rho}}}$ in terms of the function $\rho$. Interestingly, it turns out that the condition $\Theta_{21}^2\equiv 0$ implies that  on the section $\gamma_0$ all the coefficients in the expansions of all the components of $\Omega_{\Gamma_{{}_{\rho}}}$ are zero except for $\Theta_{10}^2$. Therefore, our choice of assumptions in Theorem \ref{main1} is optimal from the computational point of view. Indeed, $\Theta_{21}^2$ is the easiest coefficient to compute and, if we restrict our calculation to $\gamma_0$ (which very convenient computationally), then $\Theta_{10}^2$ cannot be replaced with any other curvature coefficient.
\end{remark}

\begin{remark}\label{twofunctions} For any real hypersurface in $\CC^3$, paper \cite{P} introduces a pair of expressions in terms of a local defining function that vanish simultaneously if and only if the hypersurface is locally CR-equivalent to $M_0$. The expressions are rather complicated, and it would be interesting to see whether in the tube case they simplify to manageable formulas that can be utilized for obtaining an alternative proof of Theorem \ref{main2}.
\end{remark}


\begin{thebibliography}{ABCD}

\bibitem[AI]{AI} Alper, J. and Isaev, A., Associated forms in classical invariant theory, preprint, available from http://arxiv.org/abs/1308.6624.
 
 \bibitem[BER]{BER} Baouendi, M. S., Ebenfelt, P. and Rothschild, L. P., {\it Real Submanifolds in Complex Space and Their Mappings}, Princeton Mathematical Series {\bf 47}, Princeton University Press, Princeton, NJ, 1999.
 
\bibitem[BS1]{BS1} Burns, D. and Shnider, S., Real hypersurfaces in complex manifolds, in {\it Several Complex Variables, Proc. Sympos. Pure Math.} {\bf XXX}, Amer. Math. Soc., Providence, R.I., 1977, pp. 141--168.

\bibitem[BS2]{BS2} Burns D., and Shnider, S., Projective connections in CR geometry, {\it Manuscripta Math.} {\bf 33} (1980/81), 1--26.

\bibitem[\v CSc]{CSc} \v Cap, A. and Schichl, H., Parabolic geometries and canonical Cartan connections, {\it Hokkaido Math. J.} {\bf 29} (2000), 453--505.

\bibitem[\v CSl]{CSl} \v Cap, A. and Slov\' ak, J., {\it Parabolic Geometries. I. Background and General Theory}, Mathematical Surveys and Monographs {\bf 154}, American Mathematical Society, Providence, RI, 2009.

\bibitem[Ca]{Ca} Cartan, \' E., Sur la g\'eometrie pseudo-conforme des hypersurfaces de l'espace de deux variables complexes: I, {\it Ann. Math. Pura Appl.} {\bf 11} (1932), 17--90; II, {\it Ann. Scuola
Norm. Sup. Pisa} {\bf 1} (1932), 333--354.

\bibitem[Ch]{Ch} Chern, S. S., On the projective structure of a real hypersurface in $\CC^{n+1}$, {\it Math. Scand.} {\bf 36} (1975), 74--82.

\bibitem[CM]{CM} Chern, S. S. and Moser, J. K., Real hypersurfaces in complex manifolds, {\it Acta Math.} {\bf 133} (1974), 219--271; erratum, {\it Acta Math.} {\bf 150} (1983), 297.

\bibitem[EI]{EI} Eastwood, M. G. and Isaev, A. V., Extracting invariants of isolated hypersurface singularities from their moduli algebras, {\it Math. Ann.} {\bf 356} (2013), 73--98.

\bibitem[EEI]{EEI} Eastwood, M., Ezhov, V. and Isaev, A., Towards a classification of homogeneous tube domains in $\CC^4$, {\it J. Differential Geom.} {\bf 68} (2004), 553--569.

\bibitem[E]{E} Ebenfelt, P., Uniformly Levi degenerate CR manifolds: the 5-dimensional case, {\it Duke Math. J.} {\bf 110} (2001), 37--80; correction, {\it Duke Math. J.} {\bf 131} (2006), 589--591.

\bibitem[EIS]{EIS} Ezhov, V. V., Isaev, A. V. and Schmalz, G., Invariants of elliptic and hyperbolic CR-structures of codimension 2, {\it Internat. J. Math.} {\bf 10} (1999), 1--52.

\bibitem[FK1]{FK1} Fels, G. and Kaup, W., CR-manifolds of dimension 5: a Lie algebra approach, {\it J. Reine Angew. Math.} {\bf 604} (2007), 47--71.

\bibitem[FK2]{FK2} Fels, G. and Kaup, W., Classification of Levi degenerate homogeneous CR-manifolds in dimension 5, {\it Acta Math.} {\bf 201} (2008), 1--82.

\bibitem[FK3]{FK3} Fels G. and Kaup, W., Classification of commutative algebras and tube realizations of hyperquadrics, preprint, available from the Mathematics ArXiv. http://arxiv.org/abs/0906.5549.

\bibitem[FK4]{FK4} Fels, G. and Kaup, W., Local tube realizations of CR-manifolds and maximal Abelian subalgebras, {\it Ann. Sc. Norm. Super. Pisa Cl. Sci. (5)} {\bf 10} (2011), 99--128. 

\bibitem[FK5]{FK5} Fels, G. and Kaup, W., Nilpotent algebras and affinely homogeneous surfaces, {\it Math. Ann.} {\bf 353} (2012), 1315--1350.

\bibitem[FIKK]{FIKK} Fels, G., Isaev, A., Kaup, W. and Kruzhilin, N., Isolated hypersurface singularities and special polynomial realizations of affine quadrics, {\it J. Geom. Analysis} {\bf 21} (2011), 767--782.

\bibitem[GM]{GM} Garrity, T. and Mizner, R., The equivalence problem for higher-codimensional CR structures, {\it Pacific. J. Math.} {\bf 177} (1997), 211--235.

\bibitem[I1]{I1} Isaev, A., {\it Spherical Tube Hypersurfaces}, Lecture Notes in Mathematics, {\bf 2020}, Springer, 2011.

\bibitem[I2]{I2} Isaev, A. V., On the affine homogeneity of algebraic hypersurfaces arising from Gorenstein algebras, {\it Asian J. Math.} {\bf 15} (2011), 631--640.

\bibitem[IZ]{IZ} Isaev, A. and Zaitsev, D., Reduction of five-dimensional uniformly Levi degenerate CR structures to absolute parallelisms, {\it J. Geom. Anal.} {\bf 23} (2013), 1571--1605.

\bibitem[J]{J} Jacobowitz, H., Induced connections on hypersurfaces in $\CC^{n+1}$, {\it Invent. Math.} {\bf 43} (1977), 109--123.

\bibitem[KZ]{KZ} Kaup, W. and Zaitsev, D., On local CR-transformation of Levi-degenerate group orbits in compact Hermitian symmetric spaces, {\it J. Eur. Math. Soc.} {\bf 8} (2006), 465--490.

\bibitem[KS1]{KS1} Kruzhilin, N. G. and Soldatkin, P. A., Affine and holomorphic equivalence of tube domains in $\CC^2$, {\it Math. Notes} {\bf 75} (2004), 623--634. 

\bibitem[KS2]{KS2} Kruzhilin, N. G. and Soldatkin, P. A., Holomorphic equivalence of tube domains in $\CC^2$, {\it Proc. Steklov Inst. Math.} {\bf 253} (2006), 90--99.

\bibitem[La]{La} Lai, H.-F., Real submanifolds of codimension two in complex manifolds, {\it Trans. Amer. Math. Soc.} {\bf 264} (1981), 331--352.

\bibitem[Lo]{Lo} Loboda, A. V., Any holomorphically homogeneous tube in $\CC^2$ has an affine-homogeneous base, {\it Siberian Math. J.} {\bf 42} (2001), 1111--1114.

\bibitem[Ma]{Ma} Matsushima, Y., On tube domains, in {\it Symmetric Spaces (Short Courses, Washington Univ., St. Louis, Mo., 1969--1970)}, Pure and Appl. Math. {\bf 8}, Dekker, New York, 1972, pp. 255--270.

\bibitem[MS]{MS} Medori, C. and Spiro, A., The equivalence problem for 5-dimensional Levi degenerate CR manifolds, preprint, available from http://arxiv.org/abs/1210.5638.

\bibitem[Me]{Me} Merker, J., Lie symmetries and CR geometry, {\it J. Math. Sci.} {\bf 154} (2008), 817--922.

\bibitem[Mi]{Mi} Mizner, R., CR structures of codimension 2, {\it J. Diff. Geom.} {\bf 30} (1989), 167--190.

\bibitem[P]{P} Pocchiola, S., Explicit absolute parallelism for 2-nondegenerate real hypersurfaces $M^5\subset\CC^3$ of constant Levi rank, preprint, available from http://arxiv.org/abs/1312.6400.

\bibitem[Sa]{Sa} Satake, I., {\it Algebraic Structures of Symmetric Domains}, $\hbox{Kan}\hat{\hbox{o}}$ Memorial Lectures {\bf 4}, Princeton Univercity Press, 1980.

\bibitem[ScSl]{ScSl} Schmalz, G. and Slov\'ak, J., The geometry of hyperbolic and elliptic CR-manifolds of codimension two, {\it Asian J. Math.} {\bf 4} (2000), 565--597; addendum, {\it Asian J. Math.} {\bf 7} (2003), 303--306.

\bibitem[ScSp]{ScSp} Schmalz, G. and Spiro, A., Explicit construction of a Chern-Moser connection for CR manifolds of codimension two, {\it Ann. Mat. Pura Appl. (4)} {\bf 185} (2006), 337--379.

\bibitem[Sh1]{Sh1} Shimizu, S., Automorphisms and equivalence of tube domains with bounded base. {\it Math. Ann.} {\bf 315} (1999), 295--320.

\bibitem[Sh2]{Sh2} Shimizu, S., Prolongation of holomorphic vector fields on a tube domain, {\it Tohoku Math. J. (2)} {\bf 65} (2013), 495--514. 

\bibitem[T1]{T1} Tanaka, N., On generalized graded Lie algebras and geometric structures I, {\it J. Math. Soc. Japan} {\bf 19} (1967), 215--254.

\bibitem[T2]{T2} Tanaka, N., On non-degenerate real hypersurfaces, graded Lie algebras and Cartan connections, {\it Japan. J. Math.} {\bf 2} (1976), 131--190.

\bibitem[T3]{T3} Tanaka, N., On the equivalence problem associated with simple graded Lie algebras, {\it Hokkaido Math. J.} {\bf 8} (1979), 23--84.

\bibitem[U]{U} Ushakov, V., The explicit general solution of trivial Monge-Amp\`ere equation, {\it Comment. Math. Helv.} {\bf 75} (2000), 125--133.

\bibitem[Y]{Y} Yang, K., {\it Exterior Differential Systems and Equivalence Problems}, Mathematics and its Applications {\bf 73}, Kluwer Academic Publishers Group, Dordrecht, 1992. 


\end{thebibliography}
\end{document}